\documentclass[11pt]{amsart}

\usepackage{amssymb,amsthm,amsmath,eucal,mathrsfs}

\usepackage{enumitem}
\usepackage{hyperref}
\usepackage{todonotes}

\advance\textheight by \topskip

\def\N{\mathbb{N}}
\def\R{\mathbb{R}}

\def\C{\mathbb{C}}

\def\cL{\mathcal{L}}

\def\PP{\mathbb{P}}
\def\EE{\mathbb{E}}

\def\cH{\mathcal{H}}

\def\loc{{\mathrm{loc}}}

\def\dive{\mbox{\rm div}\,}

\def\ran{\mathrm{ran}\,}
\def\diam{\mathrm{diam}\,}

\def\sm{\setminus}

\def\bel{\begin{equation}}
\def\eel{\end{equation}}

\def\beq{\begin{eqnarray*}}
\def\eeq{\end{eqnarray*}}

\newcommand{\sk}[2]{\langle #1,#2\rangle}

\newcommand{\n}[1]{\|#1\|}
\newcommand{\nn}[2]{\|#1\|_{#2}}

\def\sL{\mathcal{L}}

\def\eps{\varepsilon}
\def\ph{\varphi}
\def\Om{\Omega}
\def\om{\omega}
\def\wt{\widetilde}

\def\la{\lambda}

\def\al{\alpha}
\def\ga{\gamma}

\def\si{\sigma}
\def\Si{\Sigma}
\def\del{\delta}

\def\ov{\overline}

\def\emb{\hookrightarrow}

\newtheorem{prop}{Proposition}[section]

\newtheorem{remark}[prop]{Remark}
\newtheorem{corollary}[prop]{Corollary}

\newtheorem{theorem}[prop]{Theorem}

\numberwithin{equation}{section}

\title{$H^\infty$-calculus for Stokes operators on rough and on unbounded domains}

\author{Peer Christian Kunstmann and Patrick Tolksdorf}
\address{\it Karlsruhe Institute of Technology (KIT),
Institute for Analysis\\
Englerstr. 2, D -- 76128 Karlsruhe, Germany\\
e-mail: peer.kunstmann@kit.edu, patrick.tolksdorf@kit.edu}

\date{}

\begin{document}
%\maketitle

\begin{abstract}

In this article, we give an overview on known as well as new results on the boundedness of the $H^{\infty}$-calculus of the Stokes operator in rough as well as in unbounded (smoother) domains. We present a special case of an abstract comparison principle due to Kunstmann and Weis (\cite{KuW:Hinfty-Stokes}) that serves as the basis for all considerations. Subsequently, we show how this result can be applied to arrive at a bounded $H^{\infty}$-calculus for the Stokes operator. We sketch the proof for no slip boundary conditions in bounded Lipschitz domains which was given in~\cite{KuW:Hinfty-Stokes}. For unbounded domains this approach yields a shorter proof compared to previous arguments. Moreover, we further establish the boundedness of the $H^{\infty}$-calculus for the Stokes operator with Neumann type boundary conditions in bounded convex domains which is entirely new. 


\end{abstract}

\maketitle

\section{Introduction}\label{sec:intro}

Boundedness of the $H^\infty$-calculus for a sectorial operator $A$ with domain $D(A)$ in a Banach space $X$ 
%can be viewed as a manageable replacement for the 
has found numerous applications in the theory of elliptic and parabolic equations. Outside Hilbert spaces it can be viewed as a frequently available replacement for the spectral theorem that one has for non-negative self-adjoint operators. %In a Hilbert space\todo{Hier endet der Satz einfach. Soll der Anfang gel\"oscht werden?}
If the $H^\infty$-calculus for $A$ is bounded then, e.g., $A$ has bounded imaginary powers, a property which entails equality of fractional domain spaces $D(A^\al)$, $\al\in(0,1)$, with complex interpolation spaces $[X,D(A)]_\al$ or -- via the Dore-Venni theorem in a $UMD$-space $X$ -- allows to prove maximal $L^q$-regularity for $1<q<\infty$.
Other applications include maximal regularity for stochastic evolution equations or uniform stability estimates for numerical time integration schemes.

For Stokes operators with no-slip boundary conditions the first result to mention is due to Giga (\cite{Giga}) where boundedness of the imaginary powers is shown in bounded $C^\infty$-domains $\Om$ by pseudodifferential methods. This result had then been used to identify fractional domain spaces of the Stokes operator on such domains. The same method allows to obtain boundedness of the $H^\infty$-calculus for Stokes operators in $L^p_\si(\Om)$, $1<p<\infty$, on bounded $C^\infty$-domains $\Om\subseteq\R^d$. Boundedness of imaginary powers for exterior domains had been shown by Giga and Sohr in \cite{GiSo}
by localization. Noll and Saal (\cite{NoSa}) proved boundedness of the $H^\infty$-calculus for Stokes operators with no-slip boundary conditions on bounded $C^3$-domains by localization and perturbation.

Abels developed pseudodifferential methods for non-smooth, e.g., $BUC$-coefficients in $x$ and used them to obtain boundedness of the $H^\infty$-calculus for Stokes operators on a variety of unbounded domains such as, e.g., asymptotically flat layers (\cite{abels}). Pr\"uss (\cite{pruess}) showed boundedness of the $H^\infty$-calculus for generalized Stokes operators with coefficients and various boundary conditions on bounded $C^{2,1}$-domains, again by localization and perturbation, where model problems in the half space are studied via Fourier multiplier results. On the whole space, the second author of this article extended these results to generalized Stokes operators with $L^{\infty}$-coefficients in \cite{Tolksdorf_non-local} relying on the methods summarized in this article. For Stokes operators with first order boundary conditions on unbounded domains, but with a fixed wall, boundedness of the $H^\infty$-calculus had been achieved in \cite{Ku:Hinfty-1st} by means of kernel estimates for Hodge boundary conditions and perturbation.

In this paper we concentrate on the method that allows to obtain boundedness of the $H^\infty$-calculus for Stokes operators by comparison of their (fractional) domain spaces in solenoidal $L^p$-spaces with (fractional) domain spaces of corresponding Laplace operators in $L^p(\Om)^d$ via the Helmholtz projection in $L^p(\Om)^d$. This method has its origins in \cite{KKW},
in particular Theorem 8.2 there. The proof of this result contains a gap which has been fixed in \cite{KuW:erratum} by the cost of an additional assumption. Roughly speaking, comparability of certain fractional domain spaces in $L^p$ for $p=2$ and some $p_0\neq2$ allowed by interpolation to manipulate certain expressions involving Littlewood-Paley operators for the Stokes and Laplace operator at hand. Via such expressions, boundedness of the $H^\infty$-calculus for the Stokes operator could be shown in $L^p$ for $p$ between $2$ and $p_0$. 

These results were improved in \cite{KuW:Hinfty-Stokes} where the assumptions could be relaxed. The basic idea there had been to argue directly with certain (exponential) decay estimates for suitable Littlewood-Paley expressions. These decay estimates in turn can be obtained by complex interpolation. In the end, comparability of fractional domain spaces is needed only in $L^p$ for $p=2$. Since less information about Stokes and Laplace operators in $L^p$ is needed for $p\neq2$, the result could be applied to Stokes operators with no-slip boundary conditions in bounded Lipschitz domains, see \cite{KuW:Hinfty-Stokes}.

Here, we give in Section~\ref{sec:abstract} a concise version of the abstract result in \cite{KuW:Hinfty-Stokes} (see Theorem~\ref{thm:abstract}, Corollary~\ref{cor:abstract}, and Remark~\ref{rem:Cor-sa} below). In Section~\ref{sec:StokesNBC}, we apply the abstract result to establish a bounded $H^\infty$-calculus for Stokes operators with certain Neumann type boundary conditions on bounded and convex domains in $\R^d$, $d\ge2$. 
%These boundary conditions are suitable to study free boundary problems.\todo{Hier w\"urde ich nen anderen Satz schreiben, weil wir $\mu = 1$ nicht erreichen.} 
For the application we need $R$-sectorialty of these operators in a certain part of the $L^p$-scale, and we prove them here by enhancing the methods used in \cite{Tolksdorf_convex} for the proof of resolvent bounds. The other key ingredients we need had been shown in \cite{MiMoWr}. The results on $R$-sectoriality and the $H^\infty$-calculus for the Neumann type Stokes operators are entirely new.

In the rest of the paper, we review in a survey-like style other situations where boundedness of the $H^\infty$-calculus has been shown by the comparison method. More precisely, we study the Stokes operator with no-slip boundary conditions in bounded Lipschitz domains in dimensions $d\ge2$, see Section~\ref{sec:bdd-Lip}, and we review results on Stokes operators with no-slip boundary conditions in $\wt{L}^p_\si(\Om)$ for unbounded uniform $C^{1,1}$-domains (originally established in \cite{K:hinfty-stokes}) and, assuming a Helmholtz decomposition, in $L^p_\si(\Om)$ in unbounded uniform $C^3$-domains with $1<p<\infty$ (originally proven in \cite{GK}), see Section~\ref{sec:unbdd}.
The proofs we give are simpler than the original ones, and -- in the $\wt{L}^p_\si$-case -- we reobtain the original assertion of \cite[Thm. 1.1]{K:hinfty-stokes}, for which the proof given there contained a gap, discussed in \cite[Rem. 1.2]{GK}.

\section{Preliminaries}

In this section we present essential definitions and some fundamental results that we need in the following.

For a Banach space $X$ we denote by $\sL(X)$ the space of bounded linear operators on $X$ and by $X'$ the \emph{dual space} of $X$, i.e., the space of continuous linear functionals on $X$, and by $X^*$  the \emph{anti-dual space} of $X$, i.e., the space of all continuous anti-linear functionals on $X$.

\subsection{Sectorial operators}
For a linear operator $A$ in a complex Banach space $X$ we denote its spectrum by $\si(A)$, its resolvent set by $\rho(A):=\C\sm\si(A)$, and resolvent operators by $R(\la,A)=(\la-A)^{-1}$ for $\la\in\rho(A)$.

We call a linear operator $A$ in $X$ with dense domain $D(A)$ and dense range \emph{sectorial} of angle $\om\in[0,\pi)$ if 
$$
 \si(A)\subseteq\Si_\om:=\{\la\in\C\sm\{0\}:|\operatorname{arg}\la|\le\om\}\cup\{0\},
$$ 
and, for each $\theta\in(\om,\pi)$, the set 
\begin{align}\label{eq:res-set} 
 \{\la R(\la,A):\la\in \C\sm\Si_\theta\}\subseteq \sL(X)
\end{align}
is bounded. Observe that with this definition of sectoriality, $A$ is, in particular, injective.

For an angle $\theta\in(0,\pi)$ we denote by $\Si_\theta^\circ$ the interior of $\Si_\theta$ and by $H^\infty(\Si_\theta^\circ)$ the space of all functions $g:\Si_\theta^\circ\to\C$ that are bounded and holomorphic on $\Si_\theta^\circ$.
If $A$ is a sectorial operator of angle $\om\in[0,\pi)$ in $X$ and $\theta\in(\om,\pi)$ we can define a functional calculus $f\mapsto f(A)$ for functions 
$$
 f\in H^\infty_0(\Si_\theta^\circ):= \big\{g\in H^\infty(\Si_\theta^\circ):\exists\eps,c>0: |g(z)|\le c\min\{|z|^\eps,|z|^{-\eps}\}, z\in\Si_\theta^\circ \big\}
$$ 
by contour integrals. Choosing $\ga\in(\om,\theta)$ the operator $f(A)$ is given by 
$$
 f(A):=\frac1{2\pi i}\int_{\partial\Si_\ga} f(\la) R(\la,A)\,d\la,
$$
where $\partial\Si_\ga$ is parametrized such that $\Si_\ga^\circ$ lies to the left. Observe that the integral here is absolutely convergent as a Bochner integral (or as an improper Riemann integral). We mention that the formula $f(A):=(\ph(A)^{-1})^n(f\ph^n)(A)$, where $\ph(\la)=\la(1+\la)^{-2}$ and $n \in \N$, yields an extension to a functional calculus of closed but unbounded operators, which comprises in particular fractional powers $A^\al$ of $A$ for all $\al\in\C$.

If $A$ is sectorial of angle $0\le\om<\frac\pi2$, then $-A$ is the generator of a bounded analytic semigroup. More precise,
$z\mapsto e^{-zA}$ is bounded on every sector $\Si_\del$ with $\del\in(0,\frac{\pi}{2} - \om)$. 

\subsection{Bounded $H^\infty$-caculus}
Let $\theta \in (\omega , \pi)$. The operator $A$ is said to \emph{have a bounded $H^\infty(\Si_\theta^\circ)$-calculus} if there exists $C>0$ such that, for all $ f\in H^\infty_0(\Si_\theta^\circ)$, we have
\begin{align}\label{eq:hinfty-est}
 \nn{f(A)}{\sL(X)}\le C\sup\{|f(\la)|:\la\in\Si_\theta^\circ\}.
\end{align}
In this case, the functional calculus has a unique extension to all $f\in H^\infty(\Si_\theta^\circ)$ and the estimate \eqref{eq:hinfty-est} persists to this larger class of functions. We refer to, e.g., \cite[Section 9]{KuW:levico} for the construction and further properties of the $H^\infty$-calculus. 

\subsection{Bounded imaginary powers}\label{sub:BIP} We note as an immediate consequence of a bounded $H^\infty(\Si_\theta^\circ)$-calculus for $A$ that $A$ has \emph{bounded imaginary powers} $A^{it}$, $t\in\R$, with an estimate $\nn{A^{it}}{\sL(X)}\le C e^{\theta|t|}$, $t\in\R$. In particular,
domains of fractional powers of $A^\al$, $\al\in(0,1)$, can be obtained as complex interpolation spaces between $X$ and $D(A)$:
\begin{align}\label{eq:intpol-dom}
 \big[X , D(A) \big]_\al=D(A^\al),\quad \al\in(0,1),
\end{align}
see \cite[Thm.~1.15.3]{Triebel} (and \cite[Thm.~15.28]{KuW:levico} for a homogeneous version). This is one main application of a bounded $H^\infty$-calculus, but -- outside Hilbert spaces -- boundedness of the $H^\infty$-calculus is not necessary to have bounded imaginary powers.

\subsection{$R$-sectorial operators and maximal $L^q$-regularity}\label{sub:R-sec} A sectorial operator $A$ in a Banach space $X$ is called \emph{$R$-sectorial} of angle $\om\in[0,\pi)$ if $\si(A)\subseteq\Si_\om$ and, for any $\theta\in(\om,\pi)$, the set in \eqref{eq:res-set} is $R$-bounded in $\sL(X)$. Here, a set $\mathcal{T}\subseteq\cL(X)$ %in $\cL(X)$ 
is called \emph{$R$-bounded} if there exists a $C>0$ such that, for any $n\in\N$ and $x_1,\ldots,x_n\in X$, $T_1,\ldots,T_n\in\mathcal{T}$, we have the estimate
$$
 \EE\Big\| \sum_{j=1}^n\eps_j T_j x_j \Big\|_X = C \EE\Big\| \sum_{j=1}^n \eps_j x_j \Big\|_{X},
$$  
where $\eps_1,\ldots,\eps_n$ are independent and symmetric random variables with values in $\{-1,1\}$, e.g., Rademachers.
We refer to, e.g., \cite{KuW:levico} for more on this notion.

In a Banach space $X$ with the $UMD$-property, so in particular in a closed subspace of an $L^p$-space, $1<p<\infty$, we have the following two facts:

First, boundedness of the $H^\infty(\Si_\theta^\circ)$-calculus of $A$ implies that $A$ is $R$-sectorial of angle $\le\theta$ and the infimum of the $H^\infty$-angles $\theta$ equals the infimum of the $R$-sectoriality angles (see \cite[Thm. 5.3]{KaWe} and use \cite[Prop. 3.2]{KaWe}). Hence, boundedness of the $H^\infty$-calculus can be seen as \emph{qualitative} improvement of $R$-sectoriality whose \emph{quantitative} aspect (the angle) is ruled by the $R$-sectoriality angle.

Second, Lutz Weis' famous result (\cite{Weis}) characterizes, among sectorial operators of angle $<\frac\pi2$, those that have maximal $L^q$-regularity for $1<q<\infty$ as those that are $R$-sectorial of angle $<\frac\pi2$, see \cite{Weis}. Here, $A$ is said to \emph{have maximal $L^q$-regularity}, where $q\in(1,\infty)$ if, for any $f\in L^q(0,\infty;X)$ the problem 
$$
 u'(t)+Au(t)=f(t),\quad t>0,\qquad u(0)=0,
$$
has a unique mild solution $u$ which satisfies $u', Au\in L^q(0,\infty;X)$ and the estimate
$$
 \nn{u'}{L^q(0,\infty;X)}+\nn{Au}{L^q(0,\infty;X)}\le C\nn{f}{L^q(0,\infty;X)},
$$
for some $C>0$. Maximal $L^q$-regularity has found a multitude of applications in the study of nonlinear parabolic problems via fixed point arguments or the implicit function theorem.

\section{The abstract result}\label{sec:abstract}

The following result has been shown in \cite{KuW:Hinfty-Stokes}. We give here an explicit formulation of a special case for easy reference. The result can be used to transfer a bounded $H^\infty$-calculus from an operator $B$ on a scale of spaces $L^p(\Om)$, $p\in I$, to an operator $A$ on a corresponding scale of complemented subspaces $X_p$ of $L^p(\Om)$ under certain conditions on their relation. 

Here, $(\Om,\mu)$ is a $\si$-finite measure space and $I\subseteq(1,\infty)$ is an interval containing $p=2$. Strictly speaking, we have, of course, a different operator $B_p$ in $L^p(\Om)$ for each $p$ and we assume that they have consistent resolvents, i.e., $R(\la,B_{p_1})$ and $R(\la, B_{p_2})$ coincide on $L^{p_1}(\Om)\cap L^{p_2}(\Om)$ for all $p_1,p_2\in I$ and all $\la$ outside some sector containing the right half line. In this case, we shall also simply call the family $(B_p)_{p\in I}$ \emph{consistent}.  

An important special case is when $A_2$ and/or $B_2$ are self adjoint in $X_2$, $L^2(\Om)$, respectively. For duality arguments we hence work, for linear operators $T:Y\to Z$, with the \emph{adjoint} operators $T^*:Z^*\to Y^*$, which act between the anti-dual spaces $Z^*$ of $Z$ and $Y^*$ of $Y$. For a Hilbert space $H$ with scalar product $\sk{\cdot}{\cdot}$, we recall the (linear) Riesz isomorphism $H\to H^*$, $h\mapsto\sk{h}{\cdot}$. We recall for $p\in(1,\infty)$ that the anti-dual space of $L^p(\Om)$ is $(L^p(\Om))^*=L^{p'}(\Om)$ via the pairing $\sk{f}{g}_{L^p\times L^{p'}}=\int_\Om f\ov{g}\,d\mu$.

\medskip

\begin{theorem}\label{thm:abstract}
Let $p_0\in(1,\infty)\sm\{2\}$ and put $I:=[\min\{2,p_0\},\max\{2,p_0\}]$. Let $(B_p)_{p\in I}$ be a family of consistent sectorial operators in $L^p(\Om)$, $p\in I$. Assume that, for each $p\in I$, the operator $B_p$ has a bounded $H^\infty$-calculus in $L^p(\Om)$.

Let $(X_p)_{p\in I}$ be a family of closed linear subspaces $X_p\subseteq L^p(\Om)$ and let $(R_p)_{p\in I}$ and $(S_p)_{p\in I}$ be families of consistent bounded linear operators such that, for each $p\in I$, $R_p:L^p(\Om)\to X_p$, $S_p:X_p\to L^p(\Om)$ and $R_pS_p=I_{X_p}$. Let $(A_p)_{p\in I}$ be a consistent family of linear operators such that each $A_p$ is $R$-sectorial in $X_p$. 

Assume that, for some $\al_1>0>\al_2$ and some $\beta_1>0>\beta_2$ we have
\begin{align}
\label{eq:RB-incl} R_2(D(B_2^\al))&\subseteq D(A_2^\al)\quad\mbox{and}\quad
 \n{A_2^{\al}R_2f}\lesssim\n{B_2^\al f},\ \ f\in D(B_2^\al), \al\in\{\al_1,\al_2\},\\
\label{eq:SA-incl} S_2(D(A_2^\beta))&\subseteq D(B_2^\beta)\quad\mbox{and}\quad
 \n{B_2^{\beta}S_2f}\lesssim\n{A_2^\beta f},\ \ f\in D(A_2^\beta), \beta\in\{\beta_1,\beta_2\}.
\end{align}
Then, for each $p\in I\sm\{p_0\}$, $A_p$ has a bounded $H^\infty$-calculus in $X_p$.
\end{theorem}

\medskip

We mention that, for negative $\al$, $D(A_2^\al)$ equals the range of $A_2^{-\al}$. In particular, if $0\in\rho(A_2)$, then
$D(A_2^\al)=X_2$. However, the norm estimate in, e.g., \eqref{eq:RB-incl} then entails inclusions for the completions of the respective spaces. If $0\in\rho(A_2)$ and $\al>0$, then $D(A_2^\al)$ is a Banach space for the norm $\nn{A_2^\al x}{X_2}$. In certain situations this allows for applications of the closed graph theorem and its consequences.

Theorem~\ref{thm:abstract} has been proved in \cite{KuW:Hinfty-Stokes}, in principle, although not in the formulation above.
We give a proof below for convenience. We also refer to \cite[Subsection 3.2]{GT} where the approach from \cite{KuW:Hinfty-Stokes} has been sketched in some detail. Concerning the assumptions \eqref{eq:RB-incl} and \eqref{eq:SA-incl} in Theorem~\ref{thm:abstract} we remark the following.

\medskip

\begin{remark}\rm
The conditions for the negative exponents $\al_2$ and $\beta_2$ in \eqref{eq:RB-incl} and \eqref{eq:SA-incl}, respectively, can be reformulated as 
\begin{align}
 R_2^*\big(D((A_2^*)^{|\al_2|})\big)&\subseteq D\big((B_2^*)^{|\al_2|}\big), \quad\mbox{and}\quad
 \n{(B_2^*)^{|\al_2|}R_2^*f}\lesssim\n{(A_2^*)^{|\al_2|} f},\ \ f\in D((A_2^*)^{|\al_2|}), \\
 S_2^*\big(D((B_2^*)^{|\beta_2|})\big)&\subseteq D\big((A_2^*)^{|\beta_2|}\big), \quad\mbox{and}\quad
 \n{(A_2^*)^{|\beta_2|}S_2^*f}\lesssim\n{(B_2^*)^{|\beta_2|} f},\ \ f\in D((B_2^*)^{|\beta_2|}),
\end{align}
respectively. We refer to \cite[Prop. 11]{KuW:Hinfty-Stokes} for a proof.
\end{remark}

\medskip

\begin{proof}[Proof of Theorem~\ref{thm:abstract}]
First we use \cite[Prop. 10]{KuW:Hinfty-Stokes} to see that the conditions (22) and (23) in \cite[Thm. 9]{KuW:Hinfty-Stokes} are satisfied for $B_2$ and $A_2$ in $L^2(\Om)$ and $X_2$, respectively. Then we use complex interpolation to see that (22) and (23)  in \cite[Thm. 9]{KuW:Hinfty-Stokes} are satisfied for $B_p$ and $A_p$ in $L^p(\Om)$ and $X_p$, respectively, for $p$ between $2$ and $p_0$.
For the interpolation of the $R$-boundedness conditions (22) and (23) we have to use \cite[Section 1.2.4]{Triebel} and that
$L^p(\Om)^d$ and $X_p$ are $B$-convex for $1<p<\infty$, we refer to \cite[Prop. 3.7]{KKW}. 
Now, \cite[Thm. 9]{KuW:Hinfty-Stokes} yields a bounded $H^\infty$-calculus for $A_p$ in $X_p$. %\todo{Ich habe hier \"uberall $L^p_{\sigma}$ durch $X_p$ ersetzt.}
\end{proof}

\medskip

\begin{remark}\rm\label{rem:complementedXp} 
In applications $X_p$ usually is a complemented subspace in $L^p(\Om)$ and $P_p\in\sL(L^p(\Om))$ is a corresponding projection onto $X_p=\ran P_p$. Then $R_p:L^p(\Om)\to X_p$ acts as $P_p$ and $S_p=\iota_p:X_p\to L^p(\Om)$ is the inclusion map.

However, for the adjoint operators we have $R_p^*:X_p^*\to (L^p(\Om))^*$ and $P_p^*:(L^p(\Om))^*\to (L^p(\Om))^*$, in particular, $P_p^*\in\sL(L^{p'}(\Om))$ is a projection. It is easily seen that the identification $(L^p(\Om))^*=L^{p'}(\Om)$ leads to the identification $X_p^*=\ran P_p^*\subseteq L^{p'}(\Om)$ and that $R_p^*:X_p^*=\ran P_p^*\to L^{p'}(\Om)$ is the inclusion.  

For our realization $X_p^*=\ran P_p^*$, the adjoint $S_p^*: L^{p'}(\Om)\to \ran P_p^*$ of the inclusion map $S_p=\iota_p:X_p\to L^p(\Om)$ acts as $P_p^*$. 
%Observe that then the inclusions in \eqref{eq:RB-incl} and \eqref{eq:SA-incl} and interpolation imply that
%$$
%  P_2(D(B_2^\al))=D(A_2^\al)\quad\mbox{for $|\al|$ small.}
%$$ 
\end{remark}

\medskip

We further specialize to the case $\al_1=\beta_1=\al>0$ and $\al_2=\beta_2=-\beta<0$  
and study the case that $A_2$ and $B_2$ are boundedly invertible.

\medskip

\begin{remark}\rm
In the situation of Remark~\ref{rem:complementedXp}, if we have, for some $\al,\beta>0$, the conditions
\begin{align}
\label{eq:A2-B2} 
\left\{ \begin{aligned}
 P_2(D(B_2^\al))&\subseteq D(A_2^\al)\subseteq D(B_2^\al), \\ % \quad\mbox{and}\quad
 \n{A_2^{\al}P_2f} &\lesssim\n{B_2^\al f},\ \ f\in D(B_2^\al), \\ 
 \n{B_2^\al f} &\lesssim \n{A_2^\al f},\ \ f\in D(A_2^\al),
 \end{aligned} \right.
\end{align}
and %\todo{Sollte hier untendrunter noch $D((A_2^*)^{\beta})$ noch in $D((B_2^*)^{\beta})$ enthalten sein?}
\begin{align}
\label{eq:A2*-B2*} 
\left\{\begin{aligned}
P_2^*(D((B_2^*)^\beta))&\subseteq D((A_2^*)^\beta)\subseteq D((B_2^*)^\beta), \\       %\quad\mbox{and}\quad
 \n{(A_2^*)^{\beta}P_2^*f}&\lesssim\n{(B_2^*)^\beta f},\ \ f\in D((B_2^*)^\beta), \\
 \n{(B_2^*)^{\beta}f}&\lesssim\n{(A_2^*)^\beta f},\ \ f\in D((A_2^*)^\beta),
 \end{aligned}\right.
\end{align}
then \eqref{eq:RB-incl} and \eqref{eq:SA-incl} hold for $\al_1=\beta_1=\al>0$ and $\al_2=\beta_2=-\beta<0$. 

If we have, in addition, that $0\in\rho(A_2)\cap \rho(B_2)$ then condition \eqref{eq:A2-B2} is equivalent to
\begin{align}
 P_2(D(B_2^\al))=D(A_2^\al)=D(B_2^\al)\cap X_2,
\end{align}
and condition \eqref{eq:A2*-B2*} is equivalent to 
\begin{align}
 P_2^*(D((B_2^*)^\beta))=D((A_2^*)^\beta)=D((B_2^*)^\beta)\cap\ran P_2^*.
\end{align}
This can be seen by an application of the closed graph theorem.
\end{remark}

\medskip

Finally, we specialize to the case of a self-adjoint projection $P_2=P_2^*$. We formulate this as a corollary of Theorem~\ref{thm:abstract} but repeat the full list of assumptions.

\medskip

\begin{corollary}\label{cor:abstract}
Let $p_0\in(1,\infty)\sm\{2\}$ and put $I:=[\min\{2,p_0\},\max\{2,p_0\}]$. Let $(B_p)_{p\in I}$ be a family of consistent sectorial operators in $L^p(\Om)$, $p\in I$. 
Assume that, for each $p\in I$, the operator $B_p$ has a bounded $H^\infty$-calculus in $L^p(\Om)$. 

Let $(X_p)_{p\in I}$ be a family of closed linear subspaces $X_p\subseteq L^p(\Om)$ such that $X_p=\ran P_p$ for a consistent family $(P_p)_{p\in I}$ of projections $P_p\in\sL(L^p(\Om))$ and let $P_2$ be self-adjoint in $L^2(\Om)$. %and let $(R_p)_{p\in I}$ and $(S_p)_{p\in I}$ be families of consistent bounded linear operators such that, for each $p\in I$, $R_p:L^p(\Om)\to X_p$, $S_p:X_p\to L^p(\Om)$ and $R_pS_p=I_{X_p}$.
Let $(A_p)_{p\in I}$ be a consistent family of linear operators such that each $A_p$ is $R$-sectorial in $X_p$. If $0\in\rho(A_2)\cap \rho(B_2)$ and, for some $\al,\beta>0$, we have
\begin{align}
\label{eq:A2-B2-P2} P_2(D(B_2^\al))=D(A_2^\al)=D(B_2^\al)\cap X_2,   \\
\label{eq:A2*-B2*-P2} P_2(D((B_2^*)^\beta))=D((A_2^*)^\beta)=D((B_2^*)^\beta)\cap X_2,
\end{align}
then $A_p$ has a bounded $H^\infty$-calculus in $X_p$ for each $p\in I\sm\{p_0\}$.
\end{corollary}

\medskip

\begin{remark}\label{rem:Cor-sa}\rm
If, in the situation of Corollary~\ref{cor:abstract}, $B_2$ is self-adjoint in $L^2(\Om)$ and $A_2$ is self-adjoint in $X_2$ then, of course, \eqref{eq:A2*-B2*-P2} is just \eqref{eq:A2-B2-P2} for $\beta$ in place of $\al$. Hence, \eqref{eq:A2-B2-P2} for some $\al>0$ is already sufficient for the conclusion.
 
More general, if both $B_2$ and $A_2$ are associated with closed sectorial forms in $L^2(\Om)$ and $X_2$, respectively, then we have $D((B_2^*)^\al)=D(B_2^\al)$ and $D((A_2^*)^\al)=D(A_2^\al)$ for all $\al\in(0,\frac12)$ by a result due to Kato (\cite{Kato}). In this situation, \eqref{eq:A2-B2-P2} for some $\al\in(0,\frac12)$ is sufficient for the conclusion.
\end{remark}

\section{Stokes operators with Neumann type boundary conditions}\label{sec:StokesNBC}

We give a first example for an application of Corollary~\ref{cor:abstract} in case of the Stokes operator with Neumann type boundary conditions in bounded convex domains $\Om \subseteq \R^d$. More precisely, we study the operator associated to the resolvent problem 
\begin{align}
\label{Eq: Neumann resolvent}
 \left\{ \begin{aligned}
 \lambda u - \Delta u + \nabla \phi &= f && \text{in } \Om, \\
 \dive(u) &= 0 && \text{in } \Om, \\
 \{ D u + \mu [D u]^{\top} \} n - \phi n &= 0 && \text{on } \partial \Om,
\end{aligned} \right.
\end{align}
where $\mu$ is a given parameter and $n$ denotes the outward unit normal to $\partial \Omega$. This operator is studied in the scale of solenoidal $L^p$-spaces given by
\begin{align*}
 \sL^p (\Om) := \{ f \in L^p (\Om)^d \colon \dive (f) = 0 \} \qquad (1 < p < \infty).
\end{align*}
To introduce the Stokes operator subject to the Neumann type boundary condition written above, we introduce the space of solenoidal test functions
\begin{align*}
 \cH^1_{\sigma} (\Om) := \{ u \in \sL^2_{\sigma} (\Om) : \nabla u \in L^2 (\Omega)^{d \times d} \}.
\end{align*}
Given $\mu \in (-1 , 1)$ we introduce coefficients $a_{j k}^{\alpha \beta} (\mu) := \delta_{j k} \delta_{\alpha \beta} + \mu \delta_{j \beta} \delta_{k \alpha}$, where $\delta_{\alpha \beta}$ denotes Kronecker's delta. Notice that the divergence form operator with coefficients $a_{j k}^{\alpha \beta} (\mu)$ is formally given by
\begin{align*}
 \partial_j a_{j k}^{\alpha \beta} (\mu) \partial_k u_{\beta} = \Delta u_{\alpha} + \mu \partial_{\alpha} \dive(u).
\end{align*}
Thus, on solenoidal functions this operator acts merely as the Laplacian. The Stokes operator subject to the above Neumann type boundary condition is the operator $A_{\mu}$ in $\sL^2_{\sigma} (\Omega)$ associated with the sesquilinear form
\begin{align*}
 \mathfrak{a}_{\mu} : \cH^1_{\sigma} (\Omega) \times \cH^1_{\sigma} (\Omega) \to \C, \quad (u , v) \mapsto \sum_{\alpha , \beta = 1}^d \sum_{j , k = 1}^d \int_{\Om} a_{j k}^{\alpha \beta} (\mu) \partial_k u_{\beta} \overline{\partial_j v_{\alpha}} \, d x.
\end{align*}
More precisely, 
\begin{align*}
 D(A_{\mu}) &:= \{ u \in \cH^1_{\sigma} (\Om) \colon \exists f \in \sL^2_{\sigma} (\Om) \text{ s.t.\@ } \mathfrak{a}_{\mu} (u , v) = \langle f , v \rangle_{L^2} \text{ for all } v \in \cH^1_{\sigma} (\Omega) \}, \\
 A_{\mu} u &:= f \qquad (u \in D(A_{\mu})).
\end{align*}
Employing de Rham's lemma as in~\cite[Thm~6.8]{MiMoWr} one can associate to every $u \in D(A_{\mu})$ a unique pressure function $\phi \in L^2 (\Omega)$ such that
\begin{align*}
 \big\langle A_{\mu} u , v \big\rangle = \mathfrak{a}_{\mu} (u , v) - \int_{\Omega} \phi \, \overline{\dive (v)} \, d x \qquad (v \in H^1 (\Omega)^{d}).
\end{align*}
In particular, this provides the meaning of the boundary condition stated in~\eqref{Eq: Neumann resolvent}. Given $\theta \in [0 , \pi)$ and $\lambda \in \Sigma_{\theta}$, the resolvent problem with right-hand side $f \in \sL^2_{\sigma} (\Omega)$ is formulated weakly as
\begin{align}
\label{Eq: Strong resolvent Neumann}
 \lambda \int_{\Omega} u \overline{v} \, d x + \mathfrak{a}_{\mu} (u , v) - \int_{\Omega} \phi \, \overline{\dive(v)} \, d x = \int_{\Omega} f \overline{v} \, d x \qquad (v \in H^1 (\Omega)^d).
\end{align}
In the following, we denote by $\mathbb{Q}$ the orthogonal projection in $L^2 (\Omega)^d$ onto $\sL^2_{\sigma} (\Omega)$ as the \textit{Leray projection}. It is well-known that
\begin{align*}
 L^2 (\Omega)^d = \sL^2_{\sigma} (\Omega) \oplus \nabla H^1_0 (\Omega).
\end{align*}
Moreover, $\mathbb{Q}$ has a continuous extension $\mathcal{Q} : (H^1 (\Omega)^d)^* \to (\cH^1_{\sigma} (\Omega))^*$ which is given by $\mathcal{Q} v = v|_{\cH^1_{\sigma}}$ and observe that $\mathcal{Q}$ is the adjoint of the inclusion map $\cH^1_{\sigma} (\Omega) \subseteq H^1 (\Omega)^d$. \par
We will also study the Stokes resolvent problem with a right-hand side in $(\cH^1_{\sigma} (\Omega))^*$. Given $F \in L^2 (\Omega)^{d \times d}$ this right-hand side will be defined as $\mathcal{Q} \, \dive(F)$ and the corresponding resolvent problem reads
\begin{align*}
 \lambda \int_{\Omega} u \overline{v} \, d x + \mathfrak{a}_{\mu} (u , v) = - \int_{\Omega} F \cdot \overline{\nabla v} \, d x \qquad (v \in \cH^1_{\sigma} (\Omega)).
\end{align*}
Again, one can also associate a pressure function to this solution $\phi \in L^2 (\Omega)$, so that
\begin{align}
\label{Eq: Weak resolvent Neumann}
 \lambda \int_{\Omega} u \overline{v} \, d x + \mathfrak{a}_{\mu} (u , v) - \int_{\Omega} \phi \, \overline{\dive (v)} \, d x = - \int_{\Omega} F \cdot \overline{\nabla v} \, d x \qquad (v \in H^1 (\Omega)^d).
\end{align}
The boundary condition that is modelled here is given by
\begin{align*}
 \{ D u + \mu [D u]^{\top} \} n - \phi n = - F^{\top} n \quad \text{on } \partial \Om.
\end{align*}
The Lax-Milgram lemma guarantees the existence of weak solutions to both the resolvent problems~\eqref{Eq: Strong resolvent Neumann} and~\eqref{Eq: Weak resolvent Neumann}. We will often denote the solution $u$ to~\eqref{Eq: Strong resolvent Neumann} by $(\lambda + A_{\mu})^{-1} f$ and the associated pressure by $\phi = \Phi_{\lambda} f$. Moreover, the solution $u$ to~\eqref{Eq: Weak resolvent Neumann} will be denoted by $(\lambda + A_{\mu})^{-1} \mathcal{Q} \dive(F)$ and the associated pressure by $\phi = \Phi_{\lambda} \mathcal{Q} \dive(F)$. Testing the resolvent problem with $v = u$ or with $v = \nabla \Delta_D^{-1} \phi$, where $\Delta_D$ denotes the Dirichlet Laplacian on $\Omega$ yields the following basic $L^2$-result, cf.~\cite[Prop.~2.3]{Tolksdorf_convex} and~\cite[Prop.~3.1]{Tolksdorf_convex}:

\begin{prop}
\label{Prop: L2-case Neumann}
Let $\Omega \subseteq \R^d$, $d \geq 2$, be a bounded convex domain and $\theta \in [0 , \pi)$. For all $\mu \in (-1 , 1)$ we have $\Sigma_{\theta} \subseteq \rho(- A_{\mu})$. Moreover, there exists $C > 0$ depending only on $d$, $\theta$, and $\mu$ such that for all $f \in \sL^2_{\sigma} (\Omega)$ and all $\lambda \in \Sigma_{\theta}$ we have
\begin{align*}
 \lvert \lambda \rvert \| (\lambda + A_{\mu})^{-1} f \|_{\sL^2_{\sigma} (\Omega)} + \lvert \lambda \rvert^{\frac{1}{2}} \| \nabla (\lambda + A_{\mu})^{-1} f \|_{L^2 (\Omega)^{d \times d}} + \lvert \lambda \rvert^{\frac{1}{2}} \| \Phi_{\lambda} f \| \leq C \| f \|_{\sL^2_{\sigma} (\Omega)}.
\end{align*}
Moreover, there exists $C > 0$ depending only on $d$, $\theta$, and $\mu$ such that for all $F \in L^2 (\Omega)^{d \times d}$ and all $\lambda \in \Sigma_{\theta}$ we have
\begin{align*}
 \lvert \lambda \rvert^{\frac{1}{2}} \| (\lambda + A_{\mu})^{-1} \mathcal{Q} \dive(F) \|_{\sL^2_{\sigma} (\Omega)} + \| \nabla (\lambda + A_{\mu})^{-1} \mathcal{Q} \dive(F) \|_{L^2 (\Omega)^{d \times d}} + \| \Phi_{\lambda} \mathcal{Q} \dive(F) \| \leq C \| F \|_{L^2 (\Omega)}.
\end{align*}
\end{prop}

Since in Hilbert spaces uniform boundedness and $R$-boundedness are equivalent, we can read the above proposition as an $R$-boundedness result for the operator families
\begin{align*}
 f \mapsto \lambda (\lambda + A_{\mu})^{-1} f, \quad f \mapsto \lvert \lambda \rvert^{\frac{1}{2}} \nabla (\lambda + A_{\mu})^{-1} f, \quad f \mapsto \lvert \lambda \rvert^{1 / 2} \Phi_{\lambda} f \qquad (\lambda \in \Sigma_{\theta})
\end{align*}
as well as
\begin{align*}
 F \mapsto \lvert \lambda \rvert^{\frac{1}{2}} (\lambda + A_{\mu})^{-1} \mathcal{Q} \dive(F), \quad F \mapsto \nabla (\lambda + A_{\mu})^{-1} \mathcal{Q} \dive(F), \quad F \mapsto \Phi_{\lambda} \mathcal{Q} \dive(F) \qquad (\lambda \in \Sigma_{\theta}).
\end{align*}
As a first step towards an application of Corollary~\ref{cor:abstract} we establish the following $R$-boundedness result in the $L^p$-scale.

\begin{theorem}
\label{Thm: R-secoriality Neumann}
Let $\Omega \subseteq \R^d$, $d \geq 2$, be a bounded and convex domain and $r_0 > 0$ be such that $B(0 , r_0) \subseteq \tfrac{1}{2} [\Omega - \{x_0\}]$ for some $x_0 \in \Omega$. Let further $\theta \in [0 , \pi)$, $\mu \in (-1 , \sqrt{2} - 1)$, and let
\begin{align*}
 \Big\lvert \frac{1}{p} - \frac{1}{2} \Big\rvert < \frac{1}{d}\cdotp
\end{align*}
Then the families of operators
\begin{align}
\label{Eq: R-bounded resolvents}
 \big\{ \lambda (\lambda + A_{\mu})^{-1} : \lambda \in \Sigma_{\theta} \big\} &\subseteq \cL (\sL^p_{\sigma} (\Omega)),  \\
 \label{Eq: R-bounded gradients} \big\{ \lvert \lambda \rvert^{\frac{1}{2}} \nabla (\lambda + A_{\mu})^{-1} : \lambda \in \Sigma_{\theta} \big\} &\subseteq \cL (\sL^p_{\sigma} (\Omega) , L^p (\Omega)^{d \times d}), \\
 \label{Eq: R-bounded full} \big\{ \nabla (\lambda + A_{\mu})^{-1} \mathcal{Q} \dive : \lambda \in \Sigma_{\theta} \big\} &\subseteq \cL (L^p (\Omega)^{d \times d}) \quad \text{and} \\
  \label{Eq: R-bounded divergence} \big\{ \lvert \lambda \rvert^{\frac{1}{2}} (\lambda + A_{\mu})^{-1} \mathcal{Q} \dive : \lambda \in \Sigma_{\theta} \big\} &\subseteq \cL (L^p (\Omega)^{d \times d} , \sL_{\sigma}^p (\Omega))
\end{align}
are $R$-bounded. If, in addition, $p \geq 2$ we have that
\begin{align}
\label{Eq: R-bounded pressure}
  \big\{ \lvert \lambda \rvert^{\frac{1}{2}} \Phi_{\lambda} : \lambda \in \Sigma_{\theta} \big\} &\subseteq \cL (\sL^p_{\sigma} (\Omega) , L^p (\Omega))  \quad \text{and} \\
\label{Eq: R-bounded pressure, weak}
  \big\{ \Phi_{\lambda} \mathcal{Q} \dive : \lambda \in \Sigma_{\theta} \big\} &\subseteq \cL (\sL^p_{\sigma} (\Omega)^{d \times d} , L^p (\Omega))
\end{align}
are $R$-bounded. In particular, there exists a constant $C > 0$ depending only on $d$, $p$, $\theta$, $\mu$, $\diam (\Omega)$, and $r_0$ such that the $R$-bounds of all these sets of operators are bounded by $C$.
\end{theorem}

This theorem is an extension of the main result of~\cite[Thm.~1.1]{Tolksdorf_convex} which states uniform bounds of the operator families above instead of $R$-bounds. 
We want to mention that, for bounded domains with $C^{2,1}$-boundary, $R$-boundedness of corresponding operator families and maximal $L^p$-regularity for Stokes operators with Neumann type boundary conditions have been shown by Shibata and Shimizu (\cite{ShibataShimizu}) by Fourier multiplier results, perturbation, and localization.
The proof of Theorem~\ref{Thm: R-secoriality Neumann} uses the characerization of $R$-boundedness by virtue of square function estimates and then adapts the lines of the proof of~\cite[Thm.~1.1]{Tolksdorf_convex} to these square functions. Thus, we advise the reader to keep a copy of~\cite{Tolksdorf_convex} handy. We begin with the case $p > 2$.

\medbreak

\subsection*{Step~1: Reformulation into an $\ell^2$-valued boundedness estimate in the case $p > 2$}

If $V$ is a subspace of $L^p (\Omega ; \C^m)$ for some $1 < p < \infty$, $m \in \N$ and $\Omega \subseteq \R^d$ Lebesgue measurable, then there exists $C > 0$ such that for all $k_0 \in \N$ and $(f_k)_{k = 1}^{k_0} \subseteq V$ we have
\begin{align*}
 \frac{1}{C} \Big\| \sum_{k = 1}^{k_0} \eps_k (\cdot) f_k \Big\|_{L^2 (0 , 1 ; V)} \leq \Big\| \Big[ \sum_{k = 1}^{k_0} \lvert f_k \rvert^2 \Big]^{1 / 2} \Big\|_{L^p (\Omega)} \leq C \Big\| \sum_{k = 1}^{k_0} \eps_k (\cdot) f_k \Big\|_{L^2 (0 , 1 ; V)},
\end{align*}
where $\eps_1 , \dots , \eps_{k_0}$ are independent and symmetric random variables with values in $\{-1 , 1\}$. This means, that $R$-boundedness in subspaces of $L^p$ is equivalent to so-called \textit{square function estimates}. \par
 We use this fact for $p > 2$ and conclude by density that the three families of operators in~\eqref{Eq: R-bounded resolvents}-\eqref{Eq: R-bounded full} are $R$-bounded if and only if there exists $C > 0$ such that for all $k_0 \in \N$, all $\lambda_1 , \dots , \lambda_{k_0} \in \Sigma_{\theta}$ and all $f_1 , \dots , f_{k_0} \in C^{\infty}_{\sigma} (\overline{\Omega}_{\ell})$ the following estimate is valid
\begin{align*}
 \Big\| \Big[ \sum_{k = 1}^{k_0} \Big( \lvert \lambda_k (\lambda_k + A_{\mu})^{-1} f_k \rvert^2 + \lvert \lvert \lambda_k \rvert^{\frac{1}{2}} \nabla (\lambda_k + A_{\mu})^{-1} f_k \rvert^2 &+ \lvert \lvert \lambda_k \rvert^{\frac{1}{2}} \Phi_{\lambda_k} f_k \rvert^2 \Big) \Big]^{1 / 2} \Big\|_{L^p(\Omega)} \\
&\qquad\qquad \leq C \Big\| \Big[ \sum_{k = 1}^{k_0} \lvert f_k \rvert^2 \Big]^{1 / 2} \Big\|_{L^p(\Omega)}.
\end{align*}
If we define $u_k := (\lambda_k + A_{\mu})^{-1} f_k$ and $\phi_k := \Phi_{\lambda_k} f_k$, then this is equivalent to establishing that all operators of the form
\begin{align}
\label{Eq: Vector valued operator}
 T f := T (f_k)_{k = 1}^{k_0} := \left( \begin{pmatrix} \lambda_1 u_{1} \\ \vdots \\ \lambda_{k_0} u_{k_0} \end{pmatrix} , \begin{pmatrix} \lvert \lambda_1 \rvert^{1/2} \nabla u_{1} \\ \vdots \\ \lvert \lambda_{k_0} \rvert^{1/2} \nabla u_{k_0} \end{pmatrix} , \begin{pmatrix} \lvert \lambda_1 \rvert^{1/2} \phi_{1} \\ \vdots \\ \lvert \lambda_{k_0} \rvert^{1/2} \phi_{k_0} \end{pmatrix} \right)
\end{align}
extend to a uniform bounded family of operators from $\sL^p_{\sigma} (\Omega ; X)$ to $L^p (\Omega ; Y)$, where $X = \ell^2 (\C^d)$ and $Y = \ell^2(\C^d) \times \ell^2 (\C^{d \times d}) \times \ell^2 (\C)$, and where $(\lambda_k)_{k = 1}^{k_0} \subseteq \Sigma_{\theta}$ is arbitrary. Here, we regard a finite dimensional vector naturally as an element in $\ell^2$ by regarding this vector as a finite sequence and by filling up the remaining components by $0$. It is hence our task to bound the family of all such operators uniformly in these vector-valued $L^p$-spaces.

% the operator $T_{(\lambda_k)_{k = 1}^{k_0}}$ defined by
%\begin{align}
%\label{Eq: Vector valued operator}
% T_{(\lambda_k)_{k = 1}^{k_0}} : \sL^p_{\sigma} (\Omega ; \ell^2 (\C^d)) \to L^p (\Omega ; \ell^2 (\C^d)), \quad f = (f_k)_{k \in \N} \mapsto \begin{pmatrix} \lambda_1 (\lambda_1 + B_{\mu})^{-1} f_1 \\ \vdots \\ \lambda_{k_0} (\lambda_{k_0} + B_{\mu})^{-1} f_{k_0} \\ 0 \\ \vdots \end{pmatrix}
%\end{align}
%satisfies
%\begin{align}
%\label{Eq: Vector-valued boundedness}
% \| T_{(\lambda_k)_{k = 1}^{k_0}} f \|_{L^p (\Omega ; \ell^2 (\C^d))} \leq C \| f \|_{\sL^p_{\sigma} (\Omega ; \ell^2 (\C^d))}.
%\end{align}
%It is hence our task to bound the family of all such operators uniformly in these vector-valued $L^p$-spaces.

\subsection*{Step~2: Verification of the vector-valued boundedness for certain values of $p > 2$}

We employ a variant of the vector-valued version of Shen's $L^p$-extrapolation theorem which can be found, e.g., in~\cite[Thm.~4.2]{Tolksdorf_non-local}. This will be achieved by adapting~\cite[Section~6]{Tolksdorf_convex} to the vector-valued situation. In the following, $Q = Q(x_0 , r) \subseteq \R^d$ denotes a dyadic cube with center $x_0 \in \R^d$ and $\diam(Q) = r$ and $Q^*$ denotes its unique dyadic parent. Moreover, $M_{2 Q^*}$ will denote the localized maximal function on $2 Q^*$. We now state the extrapolation theorem.

\begin{theorem}
\label{Thm: Modified Shen whole space}
Let $X$ and $Y$ be Banach spaces and $\Xi \subseteq \R^d$ be measurable. Let further $2 < p < p_0$, $f \in L^2 (\Xi ; X) \cap L^p (\Xi ; X)$, and let $T$ be an operator such that $T f$ is defined and contained in $L^2 (\Xi ; Y)$. \par
Suppose that there exist constants $\iota > 1$ and $C > 0$ such that for all $\alpha > 0$ and all dyadic cubes $Q = Q(x_0 , r)$ with $r > 0$ and $x_0 \in \R^d$ the estimate
\begin{align}
\label{Eq: Generalized weak reverse Hoelder inequality IR^d}
 \begin{aligned}
 &\lvert \{ x \in Q : M_{2 Q^*} (E_0 \| T f \|_Y^2) (x) > \alpha \} \rvert \leq \frac{C}{\alpha} \int_{(2 \iota Q^*) \cap \Xi} \| f \|_X^2  \, d x \\
 &\qquad\qquad\qquad\qquad+ \frac{C \lvert Q \rvert}{\alpha^{p_0 / 2}} \bigg\{ \sup_{Q^{\prime} \supset 2 Q^*} \bigg( \frac{1}{\lvert Q^{\prime} \rvert} \int_{Q^{\prime} \cap \Xi} \big( \| T f \|_Y^2 + \| f \|_X^2 \big) \, d x \bigg)^{\frac{1}{2}} \bigg\}^{p_0}
 \end{aligned}
\end{align}
holds. Here the supremum runs over all cubes $Q^{\prime}$ containing $2 Q^*$ and $E_0$ denotes the extension to $\R^d$ by zero. \par
Then there exists a constant $K > 0$ depending on $d$, $p$, $p_0$, $\iota$, and $C$ such that
\begin{align*}
 \| T f \|_{L^p (\Xi ; Y)} \leq K \| f \|_{L^p (\Xi ; X)}.
\end{align*}
\end{theorem}

To get access to the results in~\cite[Section~4]{Tolksdorf_convex}, we need to regularize $\Omega$. For this purpose, fix $r_0 > 0$ such that $B(0 , r_0) \subseteq \tfrac{1}{2} [\Omega - \{x_0\}]$ for some $x_0 \in \Omega$ and let $(\Omega_{\ell})_{\ell \in \N}$ be the sequence of smooth, bounded and convex domains described in~\cite[Rem.~4.3]{Tolksdorf_convex}. Note, all of these domains have comparable diameter and there exists a constant $C > 0$, depending only on $r_0$, $\diam(\Omega)$ and $d$ such that for every cube $R$ with center $x_0 \in \partial \Omega_{\ell}$ and diameter $0 < r \leq 2 r_0$ we have
\begin{align}
\label{Eq: uniform d-set}
 \lvert R \cap \Omega_{\ell} \rvert \geq C r^d.
\end{align}
In the application of Theorem~\ref{Thm: Modified Shen whole space} we will take  $\Xi := \Omega_{\ell}$, $\ell \in \N$ and we define $p_0 := \frac{2d}{d - 2}$ if $d \geq 3$ and $2 < p_0 < \infty$ if $d = 2$. The operator $T$ will be a counterpart of the operator in~\eqref{Eq: Vector valued operator} on $\Omega_{\ell}$: \par
Given resolvent parameters $\lambda_1 , \dots , \lambda_{k_0} \in \Sigma_{\theta}$ and right-hand sides $f_1 , \dots , f_{k_0} \in C^{\infty}_{\sigma} (\overline{\Omega}_{\ell})$, let $u_{k , \ell} := (\lambda_k + A_{\mu , \ell})^{-1} f_k$, where $A_{\mu , \ell}$ denotes the Stokes operator with Neumann-type boundary conditions on $\Omega_{\ell}$. In addition, let $\phi_{k , \ell}$ denote the associated pressure. We now define
\begin{align}
\label{Eq: Approximate operators}
 T f := T (f_k)_{k = 1}^{k_0} := \left( \begin{pmatrix} \lambda_1 u_{1 , \ell} \\ \vdots \\ \lambda_{k_0} u_{k_0 , \ell} \end{pmatrix} , \begin{pmatrix} \lvert \lambda_1 \rvert^{1/2} \nabla u_{1 , \ell} \\ \vdots \\ \lvert \lambda_{k_0} \rvert^{1/2} \nabla u_{k_0 , \ell} \end{pmatrix} , \begin{pmatrix} \lvert \lambda_1 \rvert^{1/2} \phi_{1 , \ell} \\ \vdots \\ \lvert \lambda_{k_0} \rvert^{1/2} \phi_{k_0 , \ell} \end{pmatrix} \right)\cdotp
\end{align}
As above, we may fix $X = \ell^2 (\C^d)$ and $Y = \ell^2(\C^d) \times \ell^2 (\C^{d \times d}) \times \ell^2 (\C)$ in Theorem~\ref{Thm: Modified Shen whole space}. Now, that all notation is fixed, we can start to verify the conditions of Theorem~\ref{Thm: Modified Shen whole space}. For this purpose, we fix a dyadic cube $Q = Q(x_0 , r)$ with $2 Q^* \cap \Omega_{\ell} \neq \emptyset$. Note, if the intersection is empty, then $M_{2Q^*} (E_0 \| T f \|_Y^2)$ in~\eqref{Eq: Generalized weak reverse Hoelder inequality IR^d} would just be the zero function and the desired inequality would be trivial.

\subsection*{Case 1: We have $2 r > \sqrt{d}\, \diam(\Omega)$}

The conditions imposed on $Q^*$ and $r$ imply that for all $\ell \in \N$ we have $\Omega_{\ell} \subseteq 4 Q^*$. We use the weak-type $(1 , 1)$-estimate of the localized maximal operator and the $L^2$-boundedness of $T$ from Proposition~\ref{Prop: L2-case Neumann} to deduce
\begin{align*}
\big\lvert \big\{ x \in Q : M_{2 Q^*} (E_0 \| T f \|_Y^2) (x) > \alpha \big\} \big\rvert &\leq \frac{C_1}{\alpha} \| T f \|_{L^2 (\Omega_{\ell} ; Y)}^2 \\
&= \frac{C_1}{\alpha} \sum_{k = 1}^{k_0} \int_{\Omega_{\ell}} \Big(\lvert \lambda_k u_{k , \ell} \rvert^2 + \lvert \lvert \lambda_k \rvert^{\frac{1}{2}} \nabla u_{k , \ell} \rvert^2 + \lvert \lvert \lambda_k \rvert^{\frac{1}{2}} \phi_{k , \ell} \rvert^2 \Big) \, d x \\
&\leq \frac{C_1 C_2}{\alpha} \sum_{k = 1}^{k_0} \int_{\Omega_{\ell}} \lvert f_k \rvert^2 \, d x \\
&= \frac{C_1 C_2}{\alpha} \int_{(4Q^*) \cap \Omega_{\ell}} \| (f_k)_{k = 1}^{k_0} \|_X^2 \, d x.
\end{align*}
This proves~\eqref{Eq: Generalized weak reverse Hoelder inequality IR^d} in this particular case.

\subsection*{Case~2: We have $0 < 2 r \leq \sqrt{d}\, \diam(\Omega)$ and $2 Q^* \cap \partial \Omega_{\ell} \neq \emptyset$}
%We decompose $u_{k , \ell}$ into a local part whose $L^2$-norm is under control and a non-local part, that enjoys additional regularity properties. For this purpose, let $v \in H^1_{\sigma} (\Omega)$ and

Let $y \in (2 Q^*) \cap \partial \Omega_{\ell}$ and let $\mathcal{R} := Q (y , 4 r) \subseteq \R^d$ be the cube with center $y$ and $\diam(\mathcal{R}) = 4 r$. In particular, we have $2 Q^* \subseteq \mathcal{R}$. We split each of the functions $u_{k , \ell}$ and $\phi_{k , \ell}$ in $8 \mathcal{R} \cap \Omega_{\ell}$ into two parts using functions $v_{k , \ell}$ and $w_{k , \ell}$ as well as $\vartheta_{k , \ell}$ and $\psi_{k , \ell}$. To define these functions, let $\widetilde{A}_{\mu , \ell}$ denote the Stokes operator subject to Neumann-type boundary conditions on $(8 \mathcal{R}) \cap \Omega_{\ell}$. Notice that the restriction of $f_k$ to $(8 \mathcal{R}) \cap \Omega_{\ell}$ is still in $C^{\infty}_{\sigma} (\overline{(8 \mathcal{R}) \cap \Omega_{\ell}})$ so that the following definition is meaningful
\begin{align*}
 v_{k , \ell} := (\lambda_k + \widetilde{A_{\mu , \ell}})^{-1} R_{(8\mathcal{R}) \cap \Omega_{\ell}} f_k \qquad \text{and} \qquad w_{k , \ell} := u_{k , \ell} - v_{k , \ell}.
\end{align*}
Here $R_{(8 \mathcal{R}) \cap \Omega_{\ell}}$ denotes the restriction operator to $(8 \mathcal{R}) \cap \Omega_{\ell}$. Analogously, define the pressures $\vartheta_{k , \ell}$ associated to $v_{k , \ell}$ and $R_{(8 \mathcal{R}) \cap \Omega_{\ell}} f_k$ and $\psi_{k , \ell} := \phi_{k , \ell} - \vartheta_{k , \ell}$. Thus, in the sense of distributions we have
\begin{align*}
\left\{ \begin{aligned}
 \lambda_k v_{k , \ell} - \Delta v_{k , \ell} + \nabla \vartheta_{k , \ell} &= R_{(8 \mathcal{R}) \cap \Omega_{\ell}} f_k && \text{in } (8 \mathcal{R}) \cap \Omega_{\ell} \\
 \dive(v_{k , \ell}) &= 0 && \text{in } (8 \mathcal{R}) \cap \Omega_{\ell} \\
 \{ D v_{k , \ell} + \mu [D v_{k , \ell}]^{\top} \} n^{\ell} - \vartheta_{k , \ell} n^{\ell} &= 0 && \text{on } \partial [(8 \mathcal{R}) \cap \Omega_{\ell}]
\end{aligned} \right.
\end{align*}
and
\begin{align*}
\left\{ \begin{aligned}
 \lambda_k w_{k , \ell} - \Delta w_{k , \ell} + \nabla \psi_{k , \ell} &= 0 && \text{in } (8 \mathcal{R}) \cap \Omega_{\ell} \\
 \dive(w_{k , \ell}) &= 0 && \text{in } (8 \mathcal{R}) \cap \Omega_{\ell} \\
 \{ D w_{k , \ell} + \mu [D w_{k , \ell}]^{\top} \} n^{\ell} - \psi_{k , \ell} n^{\ell} &= 0 && \text{on } (8 \mathcal{R}) \cap \partial \Omega_{\ell}.
\end{aligned} \right.
\end{align*}
Here, $n^{\ell}$ denotes the outward unit normal vector corresponding to the set $(8 \mathcal{R}) \cap \Omega_{\ell}$. Notice that in $(8 \mathcal{R}) \cap \Omega_{\ell}$ the identities $u_{k , \ell} = v_{k , \ell} + w_{k , \ell}$ and $\phi_{k , \ell} = \vartheta_{k , \ell} + \psi_{k , \ell}$ hold and that $w_{k , \ell}$ and $\vartheta_{k , \ell}$ are in general non-zero as there is no boundary condition on the remaining boundary part $\partial [(8 \mathcal{R}) \cap \Omega_{\ell}] \setminus [(8 \mathcal{R}) \cap \partial \Omega_{\ell}]$ imposed. Let $E_0 v_{k , \ell}, E_0 \nabla v_{k , \ell}, E_0 \vartheta_{k , \ell}, E_0 w, E_0 \nabla w_{k , \ell}$, and $E_0 \psi_{k , \ell}$ denote the extensions by zero to $\R^d$. For given $\alpha > 0$ we estimate
\begin{align*}
 \bigg\lvert &\bigg\{ x \in Q : M_{2 Q^*} (E_0 \| T f \|_Y^2) (x) > \alpha \bigg\} \bigg\rvert \\
 &\leq \bigg\lvert \bigg\{ x \in Q : M_{2 Q^*} \bigg( \sum_{k = 1}^{k_0} \Big(\lvert \lambda_k E_0 v_{k , \ell} \rvert^2 + \lvert \lvert \lambda_k \rvert^{1 / 2} E_0 \nabla v_{k , \ell} \rvert^2 + \lvert \lvert \lambda_k \rvert^{1 / 2} E_0 \vartheta_{k , \ell} \rvert^2\Big)\bigg) (x) > \frac{\alpha}{4} \bigg\} \bigg\rvert \\
 &\qquad + \bigg\lvert \bigg\{ x \in Q : M_{2 Q^*} \bigg( \sum_{k = 1}^{k_0} \Big(\lvert \lambda_k E_0 w_{k , \ell} \rvert^2 + \lvert \lvert \lambda_k \rvert^{1 / 2} E_0\nabla w_{k , \ell} \rvert^2 + \lvert \lvert \lambda_k \rvert^{1 / 2} E_0 \psi_{k , \ell} \rvert^2\Big) \bigg) (x) > \frac{\alpha}{4} \bigg\} \bigg\rvert \\
 &=: \mathrm{I} + \mathrm{II}.
\end{align*}
The first term is controlled by the weak-type $(1 , 1)$-estimate of the localized maximal operator followed by the $L^2$-bounds in Proposition~\ref{Prop: L2-case Neumann} yielding
\begin{align*}
 \mathrm{I} &\leq \frac{C}{\alpha} \int_{(2 Q^*) \cap \Omega_{\ell}} \sum_{k = 1}^{k_0} \Big( \lvert \lambda_k v_{k , \ell} \rvert^2 + \lvert \lvert \lambda_k \rvert^{\frac{1}{2}} \nabla v_{k , \ell} \rvert^2 + \lvert \lvert \lambda_k \rvert^{\frac{1}{2}} \vartheta_{k , \ell} \rvert^2 \Big) \ d x \leq \frac{C}{\alpha} \int_{(32 Q^*) \cap \Omega_{\ell}} \| (f_k)_{k = 1}^{k_0} \|_X^2 \, d x,
\end{align*}
where $C > 0$ depends only on $d$, $\theta$, and $\mu$. \par
We now turn to the term $\mathrm{II}$. Usually, controlling this term requires some type of smoothing estimate in the $L^p$-scale achieved via reverse H\"older-type estimates. Recall that $p_0$ was chosen to be $p_0 := \frac{2d}{d - 2}$ if $d \geq 3$ and $p_0 > 2$ arbitrary if $d = 2$. We start by using the embedding $L^{p_0 / 2} (2 Q^*) \hookrightarrow L^{p_0 / 2 , \infty} (2 Q^*)$, the $L^{p_0 / 2}$-boundedness of the localized maximal operator and the fact $2 Q^* \subseteq \mathcal{R}$. Notice that the constants in these estimates depend only on $d$ and $q$ so that
\begin{align}
\label{Eq: Beginning of II}
 \mathrm{II} \leq \frac{C}{\alpha^{p_0 / 2}} \int_{\mathcal{R} \cap \Omega_{\ell}} \bigg( \sum_{k = 1}^{k_0} \Big( \lvert \lambda_k w_{k , \ell} \rvert^2 + \lvert \lvert \lambda_k \rvert^{\frac{1}{2}} \nabla w_{k , \ell} \rvert^2 + \lvert \lvert \lambda_k \rvert^{\frac{1}{2}} \psi_{k , \ell} \rvert^2 \Big) \bigg)^{\frac{p_0}{2}} \, d x.
\end{align}
Next, we use want to use a Sobolev inequality on the convex set $\Xi = \mathcal{R} \cap \Omega_{\ell}$. As a preparation, let us estimate a derivative of the function to which we want to apply this Sobolev inequality. The chain rule followed by the inequality of Cauchy--Schwarz yield
\begin{align*}
 &\bigg\lvert \partial_j \bigg( \sum_{k = 1}^{k_0} \Big( \lvert \lambda_k w_{k , \ell} \rvert^2 + \lvert \lvert \lambda_k \rvert^{\frac{1}{2}} \nabla w_{k , \ell} \rvert^2 + \lvert \lvert \lambda_k \rvert^{\frac{1}{2}} \psi_{k , \ell} \rvert^2 \Big) \bigg)^{\frac{1}{2}} \bigg\rvert \\
&\qquad \leq \bigg( \sum_{k = 1}^{k_0} \Big( \lvert \lambda_k w_{k , \ell} \rvert^2 + \lvert \lvert \lambda_k \rvert^{\frac{1}{2}} \nabla w_{k , \ell} \rvert^2 + \lvert \lvert \lambda_k \rvert^{\frac{1}{2}} \psi_{k , \ell} \rvert^2 \Big) \bigg)^{- \frac{1}{2}} \\
&\qquad\qquad \cdot \bigg( \sum_{k = 1}^{k_0} \Big( \lvert \lambda_k \rvert^2 \lvert w_{k , \ell} \rvert \lvert \partial_j w_{k , \ell} \rvert + \lvert \lambda_k \rvert \lvert \nabla w_{k , \ell} \rvert \lvert \partial_j \nabla w_{k , \ell} \rvert + \lvert \lambda_k \rvert \lvert \psi_{k , \ell} \rvert \lvert \partial_j \psi_{k , \ell} \rvert \Big) \bigg) \\
&\qquad \leq \bigg( \sum_{k = 1}^{k_0} \Big( \lvert \lambda_k \rvert^2 \lvert \partial_j w_{k , \ell} \rvert^2 + \lvert \lambda_k \rvert \lvert \partial_j \nabla w_{k , \ell} \rvert^2 + \lvert \lambda_k \rvert \lvert \partial_j \psi_{k , \ell} \rvert^2 \Big) \bigg)^{\frac{1}{2}}.
\end{align*}
We stress that this inequality involves only the multiplicative constant $1$. Now, the constants in Sobolev's inequality can explicitly be computed. This was done, e.g., in~\cite[Prop.~6.2]{Tolksdorf_convex}. Combined with~~\eqref{Eq: uniform d-set} this yields
\begin{align}
\label{Eq: Reverse Holder intermediate}
\begin{aligned}
 &\bigg( \int_{\mathcal{R} \cap \Omega_{\ell}} \bigg( \sum_{k = 1}^{k_0} \Big( \lvert \lambda_k w_{k , \ell} \rvert^2 + \lvert \lvert \lambda_k \rvert^{\frac{1}{2}} \nabla w_{k , \ell} \rvert^2 + \lvert \lvert \lambda_k \rvert^{\frac{1}{2}} \psi_{k , \ell} \rvert^2 \Big) \bigg)^{\frac{p_0}{2}} \, d x \bigg)^{\frac{1}{p_0}} \\
 &\qquad \leq C \bigg( r^{\frac{d}{p_0} - \frac{1}{2}} \bigg( \int_{\mathcal{R} \cap \Omega_{\ell}} \sum_{k = 1}^{k_0} \Big( \lvert \lambda_k w_{k , \ell} \rvert^2 + \lvert \lvert \lambda_k \rvert^{\frac{1}{2}} \nabla w_{k , \ell} \rvert^2 + \lvert \lvert \lambda_k \rvert^{\frac{1}{2}} \psi_{k , \ell} \rvert^2 \Big) \, dx \bigg)^{\frac{1}{2}} \\
&\qquad\qquad + r^{1 - (\frac{d}{2} - \frac{d}{p_0})} \bigg( \int_{\mathcal{R} \cap \Omega_{\ell}} \sum_{k = 1}^{k_0} \Big( \lvert \lambda_k \rvert^2 \lvert \nabla w_{k , \ell} \rvert^2 + \lvert \lambda_k \rvert \lvert \nabla^2 w_{k , \ell} \rvert^2 + \lvert \lambda_k \rvert \lvert \nabla \psi_{k , \ell} \rvert^2 \Big) \, d x \bigg)^{\frac{1}{2}} \bigg).
\end{aligned}
\end{align}
Here, the constant $C > 0$ only depends on $d$, $p_0$, $r_0$ and $\diam (\Omega)$ and is, in particular, independent of $r$, $k_0$ and $\ell$. The first term on the right-hand side is already our desired term. We thus focus on the second term. It is important to note that the sum and the integral commute by linearity. This allows us to apply the localized second-order estimate in~\cite[Prop.~4.12]{Tolksdorf_convex} as well as the Caccioppoli inequality~\cite[Lem.~6.1]{Tolksdorf_convex} term by term leading to
\begin{align*}
 &r^{1 - (\frac{d}{2} - \frac{d}{p_0})} \bigg( \int_{\mathcal{R} \cap \Omega_{\ell}} \sum_{k = 1}^{k_0} \Big( \lvert \lambda_k \rvert^2 \lvert \nabla w_{k , \ell} \rvert^2 + \lvert \lambda_k \rvert \lvert \nabla^2 w_{k , \ell} \rvert^2 + \lvert \lambda_k \rvert \lvert \nabla \psi_{k , \ell} \rvert^2 \Big) \, d x \bigg)^{\frac{1}{2}} \\
&\qquad \leq C r^{1 - (\frac{d}{2} - \frac{d}{p_0})} \bigg( \sum_{k = 1}^{k_0} \int_{(2\mathcal{R}) \cap \Omega_{\ell}} \Big( \lvert \lambda_k \rvert^3 \lvert w_{k , \ell} \rvert^2 + \frac{1}{r^2} \Big( \lvert \lvert \lambda_k \rvert^{\frac{1}{2}} \nabla w_{k , \ell} \rvert^2 + \lvert \lvert \lambda_k \rvert^{\frac{1}{2}} \psi_{k , \ell} \rvert^2 \Big) \Big) \, d x \bigg)^{\frac{1}{2}} \\
&\qquad \leq C r^{- (\frac{d}{2} - \frac{d}{p_0})} \bigg( \sum_{k = 1}^{k_0} \int_{(4\mathcal{R}) \cap \Omega_{\ell}} \Big( \lvert \lambda_k w_{k , \ell} \rvert^2 + \lvert \lvert \lambda_k \rvert^{\frac{1}{2}} \nabla w_{k , \ell} \rvert^2 + \lvert \lvert \lambda_k \rvert^{\frac{1}{2}} \psi_{k , \ell} \rvert^2 \Big) \, d x \bigg)^{\frac{1}{2}}.
\end{align*}
We combine this estimate with~\eqref{Eq: Reverse Holder intermediate} and insert the result in~\eqref{Eq: Beginning of II} to deduce the estimate
\begin{align*}
 \mathrm{II} \leq \frac{C}{\alpha^{p_0/2}} r^{- d (\frac{p_0}{2} - 1)} \bigg( \sum_{k = 1}^{k_0} \int_{(4\mathcal{R}) \cap \Omega_{\ell}} \Big( \lvert \lambda_k w_{k , \ell} \rvert^2 + \lvert \lvert \lambda_k \rvert^{\frac{1}{2}} \nabla w_{k , \ell} \rvert^2 + \lvert \lvert \lambda_k \rvert^{\frac{1}{2}} \psi_{k , \ell} \rvert^2 \Big) \, d x \bigg)^{\frac{p_0}{2}}.
\end{align*}
Finally, we add and substract $v_{k , \ell}$ and $\vartheta_{k , \ell}$ and use the $L^2$-bounds in Proposition~\ref{Prop: L2-case Neumann} to arrive at
\begin{align*}
 \mathrm{II} &\leq \frac{C \lvert Q \rvert}{\alpha^{p_0/2}} \bigg( \frac{1}{\lvert 4 \mathcal{R} \rvert} \sum_{k = 1}^{k_0} \int_{(4\mathcal{R}) \cap \Omega_{\ell}} \Big( \lvert \lambda_k u_{k , \ell} \rvert^2 + \lvert \lvert \lambda_k \rvert^{\frac{1}{2}} \nabla u_{k , \ell} \rvert^2 + \lvert \lvert \lambda_k \rvert^{\frac{1}{2}} \phi_{k , \ell} \rvert^2 + \lvert f_k \rvert^2 \Big) \, d x \bigg)^{\frac{p_0}{2}} \\
 &= \frac{C \lvert Q \rvert}{\alpha^{p_0/2}} \bigg( \frac{1}{\lvert 4 \mathcal{R} \rvert} \int_{(4\mathcal{R}) \cap \Omega_{\ell}} \Big( \| T f \|_{Y}^2 + \| (f_k)_{k = 1}^{k_0} \|_X^2 \Big) \, d x \bigg)^{\frac{p_0}{2}}.
\end{align*}
Altogether, this yields~\eqref{Eq: Generalized weak reverse Hoelder inequality IR^d}.

\subsection*{Case~3: We have $0 < 2 r \leq \sqrt{d}\, \diam(\Omega)$ and $2 Q^* \cap \partial \Omega_{\ell} = \emptyset$}

This case is treated similarly to the previous case. The only difference is that there is no need to introduce the cube $\mathcal{R}$, thus, by setting $\mathcal{R} := 2 Q^*$ in Case~2, the proof is literally the same.

\subsection*{Step~3: Conclusion of the proof in the situation $p > 2$}
Notice that the family of all operators $T$ defined as in~\eqref{Eq: Approximate operators} forms a uniformly bounded set of operators in $\cL(\sL^2_{\sigma} (\Omega_{\ell} ; X) , L^2 (\Omega_{\ell} ; Y))$ by Proposition~\ref{Prop: L2-case Neumann}. Thus, by virtue of Steps~2 and~3 and density, we conclude by Theorem~\ref{Thm: Modified Shen whole space} that for all $2 < p < 2 d / (d - 2)$ the family of all these operators $T$ is uniformly bounded from $\sL^p_{\sigma} (\Omega_{\ell} ; X)$ into $L^p (\Omega_{\ell} ; Y)$. In particular, this holds true for each of the mappings
\begin{align*}
 T_1 : f \mapsto \begin{pmatrix} \lambda_1 u_{1 , \ell} \\ \vdots \\ \lambda_{k_0} u_{k_0 , \ell} \end{pmatrix},  \quad T_2 : f \mapsto \begin{pmatrix} \lvert \lambda_1 \rvert^{1 / 2} \nabla u_{1 , \ell} \\ \vdots \\ \lvert \lambda_{k_0} \rvert^{1 / 2} \nabla u_{k_0 , \ell} \end{pmatrix}, \quad \text{and} \quad T_3 : f \mapsto \begin{pmatrix} \lvert \lambda_1 \rvert^{1 / 2} \phi_1 \\ \vdots \\ \lvert \lambda_{k_0} \rvert^{1 / 2} \phi_{k_0} \end{pmatrix}\cdotp
\end{align*}
Now, by the approximation argument carried out in the proof of~\cite[Thm.~4.4]{Tolksdorf_convex}, the uniform boundedness of these mappings carries over to the domain $\Omega$ as $\ell \to \infty$. By Step~1, this concludes the $R$-boundedness of the families in~\eqref{Eq: R-bounded resolvents},~\eqref{Eq: R-bounded gradients} and \eqref{Eq: R-bounded pressure} as operators on $\Omega$ in the case $2 < p < 2d / (d - 2)$. 

\subsection*{Step 4: The case $2d / (d + 2) < p < 2$}
By duality and Step~3, we readily conclude the $R$-boundedness of the families in~\eqref{Eq: R-bounded resolvents} and~\eqref{Eq: R-bounded divergence}. It remains to study the $R$-boundedness of the families in~\eqref{Eq: R-bounded gradients},~\eqref{Eq: R-bounded full} and~\eqref{Eq: R-bounded pressure, weak}. Concerning~\eqref{Eq: R-bounded divergence} only the case $2 < p < 2d / (d - 2)$ is missing. This will be done by repeating the argument from Steps~1-3 but for another family of operators. \par
For this purpose, let $q := 2d / (d - 2)$ if $d \geq 3$ and let $q > 2$ if $d = 2$. Let again $(\Omega_{\ell})_{\ell \in \N}$ be the sequence of bounded, convex, and smooth domains introduced in~\cite[Rem.~4.3]{Tolksdorf_convex}. Let $F_1 , \dots , F_{k_0} \in C_c^{\infty} (\Omega_{\ell} ; \C^{d \times d})$ and let $u_{1 , \ell} , \dots , u_{k_0 , \ell}$ be given by $u_{k , \ell} := (\lambda_k + A_{\mu , \ell})^{-1} \mathcal{Q} \dive (F_k)$ and let $\phi_{k , \ell}$ denote the associated pressure. Consider the operator
\begin{align}
\label{Eq: Approximate operators 2}
 S f := S (F_k)_{k = 1}^{k_0} := \left( \begin{pmatrix} \lvert \lambda_1 \rvert^{1/2} u_{1 , \ell} \\ \vdots \\ \lvert \lambda_{k_0} \rvert^{1/2} u_{k_0 , \ell} \end{pmatrix} , \begin{pmatrix}  \nabla u_{1 , \ell} \\ \vdots \\ \nabla u_{k_0 , \ell} \end{pmatrix} , \begin{pmatrix}  \phi_{1 , \ell} \\ \vdots \\ \phi_{k_0 , \ell} \end{pmatrix} \right)\cdotp
\end{align}
Notice that $S$ extends to a bounded operator from $L^2 (\Omega_{\ell} ; \ell^2(\C^{d \times d}))$ to $L^2 (\Omega_{\ell} ; \ell^2(\C^{d}) \times \ell^2(\C^{d \times d}) \times \ell^2(\C))$ by the second part of Proposition~\ref{Prop: L2-case Neumann} and that its operator norm is bounded by a constant depending merely on $d$, $\mu$, and $\theta$. Now, the assumptions of Theorem~\ref{Thm: Modified Shen whole space} are verified analogously to Cases~1-3 above. Thus, each of the families
\begin{align*}
 \big(\lvert \lambda \rvert^{1 / 2} (\lambda + A_{\mu , \ell})^{-1} \mathcal{Q} \dive \big)_{\lambda \in \Sigma_{\theta}}, \quad \big( \nabla (\lambda + A_{\mu , \ell})^{-1} \mathcal{Q} \dive \big)_{\lambda \in \Sigma_{\theta}} \quad \text{and} \quad \big(\Phi_{\lambda} \mathcal{Q} \dive\big)_{\lambda \in \Sigma_{\theta}}
\end{align*}
gives rise to $R$-bounded families of operators on $L^r (\Omega_{\ell})$ for each $2 < r < q$. The approximation argument carried out in the proof of~\cite[Thm.~4.4]{Tolksdorf_convex}, implies the $R$-boundedness of these mappings on the domain $\Omega$. By duality, the remaining $R$-boundedness properties stated in Theorem~\ref{Thm: R-secoriality Neumann} follow. This finishes the proof of Theorem~\ref{Thm: R-secoriality Neumann}. \qed
%of the mapping $T_{\lambda}^1$ from Case~1 and from the boundedness properties of $S_{\lambda}^1$ that there exists $C > 0$ such that for all $\lambda \in \S_{\theta}$ and all $f \in \cL^p_{\sigma} (\Omega)$ it holds
%\begin{align}
%\label{Eq: Estimate for small p}
% \| \lambda (\lambda + B_{\mu})^{-1} f \|_{\cL^p_{\sigma} (\Omega)} + \lvert \lambda \rvert^{1 / 2} \| \nabla (\lambda + B_{\mu})^{-1} f \|_{\L^p (\Omega ; \IC^{d^2})} \leq C \| f \|_{\cL^p_{\sigma} (\Omega)}.
%\end{align}
%The estimate on $\nabla (\lambda + \cB_{\mu})^{-1} \divergence$ follows from the boundedness of $S_{\lambda}^2$ and duality.

\bigskip

In order to apply Corollary~\ref{cor:abstract} and Remark~\ref{rem:Cor-sa} we need suitable operators $B_p$ in $L^p(\Om)^d$. Here we take simply the negative Neumann Laplacian $B$ on $\Om$ which is associated with the closed and symmetric sesquilinear form 
$$
 \mathfrak{b}(u,v):=\int_\Om \nabla u\cdot\ov{\nabla v}\,dx,\quad u,v\in H^1(\Om)^d,
$$
in $L^2(\Om)^d$. Then $B$ is self-adjoint and non-negative in $L^2(\Om)^d$, and there is a consistent family $(B_p)_{p\in(1,\infty)}$ in $L^p(\Om)^d$, where each operator $B_p$ has a bounded $H^\infty$-calculus in $L^p(\Om)^d$, $p\in(1,\infty)$. 
Since $B_2=B$ is self-adjoint and non-negative in $L^2(\Om)^d$ we have $D(B_2^{1/2})=H^1(\Om)^d$ and  
\begin{equation}\label{eq:NLap-frac}
 D(B_2^{s/2})= \big[ L^2(\Om)^d,H^1(\Om)^d \big]_{s}=H^{s}(\Om)^d,\quad s\in(0,1).
\end{equation} 
Now we need two results due to Mitrea, Monniaux, and Wright. First we quote the following on the mapping properties of the projection $\mathbb{Q}$, which is part of \cite[Thm. 7.5]{MiMoWr}.

\begin{prop}\label{prop:MMW-HHdec}
For every Lipschitz domain $\Om\subseteq\R^d$ we have 
$$
 H^s(\Om)^d=(\cL^2_\si(\Om)\cap H^s(\Om)^d)\oplus\nabla H^{s+1}_0(\Om), \quad s\in(0,1/2).
$$
\end{prop}

The second result relates the domains of fractional powers of the operators $A_\mu$ in $\cL_\si^2(\Om)$ to solenoidal Sobolev spaces and is part of \cite[Thm. 9.1]{MiMoWr}.

\begin{prop}\label{prop:MMW-frac}
If $\Om\subseteq\R^d$ is a Lipschitz domain and $\mu\in(-1,1]$ then
$$
 D(A_\mu^{s/2})=H^s(\Om)^d\cap \cL^2_\si(\Om),\quad s\in[0,1].
$$
\end{prop}

The operators $B$ and $A_\mu$ have a non-trivial kernel and thus a non-dense range in $L^2(\Om)^d$ and $\cL^2_\si(\Om)$, respectively. For simplicity, we thus state the following result on the $H^\infty$-calculus for translates $\eps+A_\mu$ $\eps>0$, of $A_\mu$.

\begin{theorem}\label{thm:Hinfty-NeumannStokes}
Let $d\ge2$, $\Omega \subseteq \R^d$ be a bounded and convex domain. Let $\mu \in (-1 , \sqrt{2} - 1)$ and
\begin{align*}
 \Big\lvert \frac{1}{p} - \frac{1}{2} \Big\rvert < \frac{1}{d}\cdotp
\end{align*}
For each $\eps>0$ the operator $\eps+A_\mu$ has a bounded $H^\infty$-calculus in $\sL^p_\si(\Om)$. 
\end{theorem}

\begin{proof}
Observe that the descriptions of the fractional domain spaces in \eqref{eq:NLap-frac} and Proposition~\ref{prop:MMW-frac} also hold for $\eps+B$ and $\eps+A_\mu$, respectively. Combine these with Proposition~\ref{prop:MMW-HHdec}, Theorem~\ref{Thm: R-secoriality Neumann}, Corollary~\ref{cor:abstract}, and Remark~\ref{rem:Cor-sa}. 
\end{proof}

\section{Stokes operators with no-slip conditions in bounded Lipschitz domains}\label{sec:bdd-Lip}

Let $\Om\subseteq\R^d$ be a bounded Lipschitz domain. For $p\in(1,\infty)$, we denote by $L^p_\si(\Om)$ the closure in $L^p(\Om)^d$ of the space
$$
 C^\infty_{c,\si}(\Om):=\{\ph\in C^\infty_c(\Om)^d: \dive \ph=0\ \mbox{in $\Om$}\}.
$$
Then $L^p_\si(\Om)=\{f\in L^p(\Om)^d: \dive f=0, \nu\cdot f|_{\partial\Om}=0\,\}$, where $\nu\cdot f|_{\partial\Om}$ is defined in the sense of traces for $f\in L^p(\Om)^d$ with $\dive f\in L^p(\Om)$.

\subsection{The operator and spaces}
We define $H^1_{0,\si}(\Om)$ as the closure of $C_{c,\si}^{\infty}(\Om)$ in $H^1(\Om)^d$. Then we have 
\begin{align*}
 H^1_{0,\si}(\Om)&=\{u\in L^2_\si(\Om): \nabla u\in L^2(\Om)^{d\times d}, u|_{\partial\Om}=0\}  \\
 &=\{u\in L^2(\Om)^d:\dive u=0, \nabla u\in L^2(\Om)^{d\times d}, u|_{\partial\Om}=0\}.
\end{align*}
The Stokes operator in $\Om$ with no-slip boundary conditions is the operator $A$ that is associated with the symmetric and closed  sesquilinear form
$$
 \mathfrak{a}(u,v):=\int_\Om \nabla u\cdot\ov{\nabla v}\,dx,\quad u,v\in H^1_{0,\si}(\Om),
$$
in $L^2_\si(\Om)$. This means for $u,f\in L^2_\si(\Om)$ that
$$
 u\in D(A)\ \mbox{and}\ Au=f\ \Longleftrightarrow\ 
 u\in H^1_{0,\si}(\Om)\ \mbox{and}\ \forall v\in H^1_{0,\si}(\Om): \mathfrak{a}(u,v)=\int_\Om f\cdot \ov{v}\,dx.
$$
In other words, in the Gelfand triple $H^1_{0,\si}(\Om)\emb L^2_\si(\Om)\emb (H^1_{0,\si}(\Om))^*$, the operator $A$ is the part in $L^2_\si(\Om)$ of the operator
$$
 \mathcal{A}:H^1_{0,\si}(\Om)\to (H^1_{0,\si}(\Om))^*, \quad u\mapsto \mathfrak{a}(u,\cdot).
$$
Since $\mathfrak{a}$ is symmetric, the operator $A$ is self-adjoint in $L^2_\si(\Om)$.
We shall compare the operator $A$ to the negative Dirichlet Laplacian $B$ on $\Om$, which in $L^2(\Om)^d$ is associated with the symmetric and closed sesquilinear form 
$$
 \mathfrak{b}(u,v):=\int_\Om \nabla u\cdot\ov{\nabla v}\,dx,\quad u,v\in H^1_{0}(\Om)^d,
$$
which means for $u,f\in L^2(\Om)^d$ that
$$
 u\in D(B)\ \mbox{and}\ Bu=f\ \Longleftrightarrow\ 
 u\in H^1_{0}(\Om)^d\ \mbox{and}\ \forall v\in H^1_{0}(\Om)^d: \mathfrak{b}(u,v)=\int_\Om f\cdot \ov{v}\,dx.
$$
So, in the Gelfand triple $H^1_0(\Om)^d\emb L^2(\Om)^d\emb (H^1_0(\Om)^d)^*$, the operator $B$ is the part in $L^2(\Om)^d$ of the operator
$$
 \mathcal{B}:H^1_0(\Om)^d \to (H^1_0(\Om)^d)^*, \quad u\mapsto \mathfrak{b}(u,\cdot),
$$
which acts on $u\in H^1_0(\Om)^d$ as $\mathcal{B}u=-\Delta u$ in distributional sense. Hence we have
$$
 D(B)=\{u\in H^1_0(\Om)^d: \Delta u\in L^2(\Om)^d\},\quad Bu=-\Delta u.
$$
Since $\mathfrak{b}$ is symmetric, the operator $B$ is self-adjoint in $L^2(\Om)^d$. 

Denoting by $\PP$ the orthogonal projection in $L^2(\Om)^d$ onto $L^2_\si(\Om)$ (the \emph{Helmholtz projection}) we have 
$$
 L^2(\Om)^d=L^2_\si(\Om) \oplus G^2(\Om),
$$
where $G^2(\Om)=\{\nabla p\in L^2(\Om)^d: p\in L^2(\Om)\}=\nabla H^1(\Om)$ is the space of gradients in $L^2(\Om)^d$. Then 
$\PP$ has a continuous extension $\mathcal{P}: (H^1_0(\Om)^d)^* \to (H^1_{0,\si}(\Om))^*$, which is given by $\mathcal{P}\phi:=\phi|_{H^1_{0,\si}(\Om)}$. The adjoint operator of $\mathcal{P}$ is the inclusion map $J:H^1_{0,\si}(\Om)
\to H^1_0(\Om)^d$. A rather obvious relation of the Stokes operator to the negative Dirichlet Laplacian is then given by
\begin{align}\label{eq:Stokes-DLap}
 \mathcal{A}=\mathcal{P}\mathcal{B}J,
\end{align}
see \cite[Prop. 2.3]{Monniaux}.
We remark that, in a canonical way, 
$$
 (H^1_{0,\si}(\Om))^*\equiv (H^1_0(\Om)^d)^*/ \{\phi : \mathcal{P}\phi|_{H^1_{0,\si}(\Om)}=0\} , \qquad
  \{\phi\in (H^1_0(\Om)^d)^* : \mathcal{P}\phi|_{H^1_{0,\si}(\Om)}=0\} = \nabla L^2(\Om),
$$
and thus elements of $(H^1_{0,\si}(\Om))^*$ may be regarded as equivalence classes $\phi+\nabla L^2(\Om)$, $\phi\in (H^1_0(\Om)^d)^*$. Taking the part in $L^2_\si(\Om)$ of the right hand side of \eqref{eq:Stokes-DLap} thus leads to the following description of the Stokes operator $A$ in $L^2_\si(\Om)$, see \cite[Thm. 4.7]{MM:JFA}:
$$
D(A)=\{u\in H^1_{0,\si}(\Om): \exists g\in L^2(\Om): -\Delta u+\nabla g\in L^2_\si(\Om) \}, \quad Au=-\Delta u+\nabla g.
$$
For $p\in(1,\infty)$ with $|\frac1p-\frac12|\le\frac16+\eps$, where $\eps$ depends on $\Om$, it had been shown in \cite{Fabes-Mendez-Mitrea} that the Helmholtz projection $\PP$ in $L^2(\Om)^d$ extends consistently to a bounded projection $\PP_p$ in $L^p(\Om)^d$ and that a  corresponding Helmholtz decomposition 
\begin{align}\label{eq:HH-dec}
 L^p(\Om)^d=L^p_\si(\Om)\oplus \nabla W^{1,p}(\Om)
\end{align}
holds. This range is known to be optimal for $d\ge3$. For bounded Lipschitz domains $\Om\subseteq\R^2$, boundedness of $\PP_p$ and the Helmholtz decomposition of $L^p(\Om)^d$ hold for the larger range of $p$ given by $|\frac12-\frac1p|\le\frac14+\eps$ (\cite{DMitrea}).
 
For $d\ge3$ it had been shown by Z. Shen in \cite{Shen2012} that, for $p\in(1,\infty)$ with $|\frac1p-\frac12|\le\frac1{2d}+\eps$, where $\eps$ depends on $\Om$, resolvents of the Stokes operator $A$ in $L^2_\si(\Om)$ extend by consistency to resolvents of a sectorial operator in $L^p_\si(\Om)$ (see \cite[Thm. 1.1]{Shen2012}) and that this operator, the Stokes operator $A_p$ in $L^p_\si(\Om)$, is given by
$$
 D(A_p)=\{u\in W^{1,p}_0(\Om)^d: \dive u=0,  \exists g\in L^p(\Om): -\Delta u+\nabla g\in L^p_\si(\Om) \}, 
 \quad A_pu=-\Delta u+\nabla g,
$$
see \cite[Cor. 1.2 and Rem. 6.4]{Shen2012}. We further mention, that resolvent bounds for the the Stokes operator in planar Lipschitz domains were recently established in $L^{\infty}$ by Geng and Shen (\cite{Geng_Shen}). 

\subsection{$R$-sectoriality and bounded $H^\infty$-calculus}
Building upon the arguments in \cite{Shen2012}, the following has been independently shown in \cite{KuW:Hinfty-Stokes} and \cite{Tolksdorf-Diss}.

\begin{prop}\label{prop:Lip-dge3}
Let $d\ge3$ and $\Om\subseteq\R^d$ be a bounded Lipschitz domain. For  $\theta\in(\frac\pi2,\pi)$ there exists $\eps>0$ such that, for $p\in(1,\infty)$ satisfying $|\frac12-\frac1p|\le \frac1{2d}+\eps$, the set 
$$
 \{(|\la|+1)(\la+A_p)^{-1}:\la\in\Si_\theta\}\subseteq \sL(L^p_\si(\Om))
$$
is $R$-bounded.
\end{prop}

For $d=3$, the range of $p$ in Proposition~\ref{prop:Lip-dge3} matches the range of $p$ for the validity of the Helmholtz decomposition \eqref{eq:HH-dec}. For $d>3$ the range of $p$ is smaller, and it is unknown if this is just due to the method of proof.

For an application of Corollary~\ref{cor:abstract} and Remark~\ref{rem:Cor-sa} we now only need to relate domains of some fractional powers of $A$ and $B$. The following is covered by \cite[Prop. 2.16]{MM:JFA}, which is much more general.

\begin{prop}\label{prop:fract-HH}
Let $\Om\subseteq\R^d$ be a bounded Lipschitz domain. For $|s|<1/2$ the Helmholtz projection $\PP$ is bounded in $H^s(\Om)^d$ and induces the topological direct sum decomposition
$$
 H^s(\Om)^d=H^s_\si(\Om)\oplus \nabla H^{s+1}(\Om)
$$ 
where $H^s_\si(\Om)=H^s(\Om)^d\cap L^2_\si(\Om)$.
\end{prop}

We also refer to \cite[Prop. 14]{KuW:Hinfty-Stokes} where an argument is sketched that follows the proof of \cite[Thm. 11.1]{Fabes-Mendez-Mitrea}. 

The following, which is part of \cite[Prop. 15]{KuW:Hinfty-Stokes}, identifies fractional domain spaces of $A$ and $B$ as the fractional Sobolev spaces in Proposition~\ref{prop:fract-HH}.  

\begin{prop}\label{prop:Lip-fract-AB}
Let $\Om\subseteq\R^d$ be a bounded Lipschitz domain. For $0<s<1/2$ we have
$$
 D(B_2^{s/2})=H^s(\Om)^d\quad\mbox{and}\quad D(A_2^{s/2})=H^s(\Om)^d\cap L^2_\si(\Om).
$$
\end{prop}

Combination of Propositions~\ref{prop:fract-HH} and \ref{prop:Lip-fract-AB} yields \eqref{eq:A2-B2-P2} and, respecting Remark~\ref{rem:Cor-sa}, we can apply Corollary~\ref{cor:abstract} to obtain the following. 

\begin{theorem}\label{thm:Lip-dge3}
Let $d\ge3$ and $\Om\subseteq\R^d$ be a bounded Lipschitz domain. Then there exists $\eps>0$ such that, for $p\in(1,\infty)$ satisfying $|\frac12-\frac1p|\le\frac1{2d}+\eps$, the Stokes operator $A_p$ has a bounded $H^\infty$-calculus in $L^p_\si(\Om)$.
\end{theorem}

In \cite{GT}, the corresponding counterpart for $d=2$ has been shown, where the range for $p$ is larger and matches the range for validity of the $L^p$-Helmholtz decomposition \eqref{eq:HH-dec}, see \cite[Thm. 1.5]{GT}.

\begin{theorem}\label{thm:Lip-d2}
Let $\Om\subseteq\R^2$ be a bounded Lipschitz domain. For $p\in(1,\infty)$ satisfying $|\frac12-\frac1p|\le\frac14+\eps$ the Stokes operator $A_p$ has a bounded $H^\infty$-calculus in $L^p_\si(\Om)$.
\end{theorem}

The proof of Theorem~\ref{thm:Lip-d2} is done in the same way as the proof of Theorem~\ref{thm:Lip-dge3}. We just have to replace the assertion of Proposition~\ref{prop:Lip-dge3} by the following counterpart for $d=2$ (see \cite[Subsection 3.1]{GT}).

\begin{prop}\label{prop:Lip-d2}
Let $\Om\subseteq\R^2$ be a bounded Lipschitz domain. For $\theta\in(\frac\pi2,\pi)$ there exists $\eps>0$ such that, for $p\in(1,\infty)$ satisfying $|\frac12-\frac1p|\le \frac1{4}+\eps$,  the set 
$$
 \{(|\la|+1)(\la+A_p)^{-1}:\la\in\Si_\theta\}\subseteq \sL(L^p_\si(\Om))
$$
is $R$-bounded.
\end{prop}

\subsection{Domains of fractional powers}
Boundedness of the $H^\infty$-calculus in Theorems~\ref{thm:Lip-dge3} and \ref{thm:Lip-d2} has been used to identify fractional domain spaces $D(A_p^\theta)$ in \cite[Thm. 1.1]{Tolksdorf} for $d\ge3$ and $\theta=\frac12$ and 
\cite[Thm. 1.3]{GT} for $d=2$ and more general $\theta$. Those results extended the result contained in \cite[Thm.~16]{KuW:Hinfty-Stokes}, which was obtained by abstract means.
It is remarked in \cite[p. 256]{GT} that the same approach as in \cite{GT} works in also in dimensions $d\ge3$. We present the precise formulation here and sketch elements of the proof.

\begin{theorem}\label{thm:fract-dom-Lip}
Let $\Om\subseteq\R^d$ be a bounded Lipschitz domain where $d\ge2$. There exists $\eps>0$ such that, for $p$ satisfying $|\frac12-\frac1p|<\frac1{2d}+\eps$ and $\theta\in(0,1)$, we have
\begin{align}\label{eq:fract-dom-Lip-incl}
                 H^{2\theta,p}_{0,\si}(\Om)\subseteq D(A_p^\theta).
\end{align}
Moreover, there exists $\del\in(0,1]$ such that, for $q$ and $\theta$ satisfying in addition
\begin{align}\label{eq:fract-dom-Lip-cond}
 \theta<\frac12+\frac1{2p}\ \mbox{if}\ \frac12-\frac1p\le\frac\del2\quad\mbox{or}\quad
 \theta<\frac{3 - d}{4}+\frac{d}{2p}+\frac{(d - 1) \delta}{4}\ \mbox{if}\ \frac12-\frac1p>\frac\del2,
\end{align}
we have (with equivalent norms)
\begin{align}\label{eq:fract-dom-Lip}
       H^{2\theta,p}_{0,\si}(\Om)= D(A_p^\theta).
\end{align}
Here, the spaces $H^{s , p}_{0 , \sigma} (\Omega)$ are defined as the spaces $V^{s , p} (\Omega)$ in~\cite{MM:JFA}. 
\end{theorem}

\begin{proof}[Sketch of proof]
The first step is to show, for $p$ with  $|\frac12-\frac1p|<\frac1{2d}+\eps$, the inclusion $W^{2,p}_{0,\si}(\Om)\subseteq D(A_p)$, see \cite[Lem. 2.5]{Tolksdorf} for $d\ge3$ and (with the same proof) \cite[Lem. 3.8]{GT} for $d=2$.

Then Theorems~\ref{thm:Lip-dge3} and \ref{thm:Lip-d2} imply via Subsection~\ref{sub:BIP} that
$$
 \big[ L^p_\si(\Om),W^{2,p}_{0,\si}(\Om) \big]_\theta\subseteq D(A_p^\theta),\quad \theta\in(0,1).
$$
By \cite[Thm. 2.12]{MM:JFA} the space on the left hand side equals $H^{2\theta,p}_{0,\si}(\Om)$, and \eqref{eq:fract-dom-Lip-incl} is proved.

The proof of \eqref{eq:fract-dom-Lip} under the conditions \eqref{eq:fract-dom-Lip-cond} relies on results by Mitrea and Wright (see \cite[Thm. 10.6.2]{MWr}) on solvability of the stationary Stokes system in the scale of Besov- and Triebel-Lizorkin spaces. The latter comprises the spaces $H^s_p$ that are relevant here. 

In order to show boundedness of 
$$
 A_p^{-\theta}:L^p_\si(\Om)\to H^{2\theta,p}_{0,\si}(\Om)
$$ 
it suffices by duality to show boundedness 
$$
 A_{p'}^{-\theta}\PP_{p'}:H^{-2\theta,p'}(\Om)^d\to L^{p'}_\si(\Om).
$$ 
Via \eqref{eq:fract-dom-Lip-incl} this can be reduced to showing boundedness of 
$$
 A_{p'}^{-1}\PP_{p'}:H^{-2\theta,p}(\Om)^d \to H^{2-2\theta,p'}_{0,\si}(\Om).
$$
Here, the results from \cite{MWr} on the Stokes system come into play, for details we refer to \cite{GT}.
\end{proof}

\section{Stokes operators on unbounded domains}\label{sec:unbdd}

In this section we consider Stokes operators with no-slip boundary conditions on unbounded domains. In contrast to the two previous sections the boundary of $\Om$ will be relatively smooth. The main issue is that unboundedness of the domain may cause the Helmholtz decomposition to fail in $L^p(\Om)^d$ for certain $p\neq2$, even if the boundary $\partial\Om$ is smooth, see \cite{MaBo}. Two different ideas have been used in the theory for Stokes operators for such domains. 
There is the $\wt{L}^p$-theory by Farwig, Kozono, and Sohr (see \cite{FKS:ArchMath}, \cite{FKS:maxreg}), and there is a by \emph{relative $L^p$-theory} Gei{\ss}ert, Heck, Hieber, and Sawada (see \cite{GHHS}), where a Helmholtz decomposition for $L^p(\Om)^d$ is \emph{assumed}. Here, it is assumed that $\Om\subseteq\R^d$ is a uniform $C^{1,1}$-domain or a uniform $C^3$-domain, respectively.

\subsection{$\wt{L}^p$-theory}
 Farwig, Kozono, and Sohr established a theory for the Stokes operator in $\wt{L}^p$-spaces, where   
$$
 \wt{L}^p(\Om):=\left\{\begin{aligned}
                           &L^p(\Om)\cap L^2(\Om), && 2\le p<\infty,\\
                            &L^p(\Om)+ L^2(\Om), && 1<p<2.\\
                           \end{aligned}
                           \right.
$$
If $\Om\subseteq\R^d$ is a uniform $C^1$-domain, they showed in \cite{FKS:ArchMath} the Helmholtz decomposition
$$
 \wt{L}^p(\Om)^d=\wt{L}^p_\si(\Om)\oplus \wt{G}^p(\Om),\quad 1<p<\infty,
$$
with corresponding projecton operator $\wt{\PP}_p$, where     %and the corresponding spaces of solenoidal and gradient vector fields
$$
 \wt{L}^p_\si(\Om):=\left\{\begin{aligned}
                           &L^p_\si(\Om)\cap L_\si^2(\Om), && 2\le p<\infty,\\
                            &L^p_\si(\Om)+ L_\si^2(\Om), && 1<p<2,\\
                           \end{aligned}
                           \right.
\qquad
 \wt{G}^p(\Om):=\left\{\begin{aligned}
                           &G^p(\Om)\cap G^2(\Om), && 2\le p<\infty,\\
                            &G^p(\Om)+ G^2(\Om), && 1<p<2,\\
                           \end{aligned}
                           \right.
$$
and $G^p(\Om)$ is now given as $G^p(\Om)=\{\nabla g\in L^p(\Om)^d: g\in L^p_\loc(\ov{\Om})\}$.
Denoting by $D^p(\Om):=W^{2,p}(\Om)\cap W^{1,p}_0(\Om)$ the domain of the Dirichlet Laplacian in $L^p(\Om)$, and defining in an analogous way spaces $\wt{D}^p(\Om)$, $1<p<\infty$, the Stokes operator $\wt{A}_p$ with no-slip boundary conditions in $\wt{L}^p_\si(\Om)$ is defined by
$$
 \wt{A}_pu:=-\wt{\PP}_p\Delta u,\quad u\in D(\wt{A}_p):=\wt{D}^p(\Om)^d\cap \wt{L}^p_\si(\Om).
$$
The following has been shown in \cite[Thm. 1.3 and 1.4]{FKS:maxreg}. 

\begin{theorem}\label{thm:qtilde-maxreg}
Let $\Om\subseteq\R^d$ be a uniform $C^{1,1}$-domain. For $1<p<\infty$ and $\del>0$, the Stokes operator $\del+\wt{A}_p$ has maximal $L^q$-regularity in $\wt{L}^p_\si(\Om)$ for $1<q<\infty$, and hence is $R$-sectorial in $\wt{L}^p_\si(\Om)$ of angle $<\frac\pi2$.
\end{theorem}

This has been used in the proof of the following, which is \cite[Thm. 1.1]{K:hinfty-stokes}.

\begin{theorem}\label{thm:qtilde-hinfty}
Let $\Om\subseteq\R^d$ be a uniform $C^{1,1}$-domain. For $1<p<\infty$ and $\del>0$, the Stokes operator $\del+\wt{A}_p$ has a bounded $H^\infty$-calculus in $\wt{L}^p_\si(\Om)$.
\end{theorem}

The proof given in \cite{K:hinfty-stokes} was based on \cite[Thm. 8.2]{KKW}, whose proof contains a gap. The gap has been closed in \cite{KuW:erratum} by the cost of an additional assumption on the Helmholtz decomposition for $p\neq2$. For the application, this additional assumption required
more regularity of $\Om$, namely that $\Om$ is a uniform $C^{2,\al}$-domain for some $\al\in(0,1]$, we refer to the discussion on this issue in \cite{GK}.

Here, we can use the analog for $\wt{L}^q$-spaces of Corollary~\ref{cor:abstract} and Remark~\ref{rem:Cor-sa} instead of the result from \cite{KuW:erratum}. Indeed, the proof in Section~\ref{sec:abstract} is based on the fact that the $L^p$-scale is a complex interpolation scale. This persists to the $\wt{L}^p$-scale, see \cite[Lem. 4.1]{K:hinfty-stokes}, which in turn is based on results in \cite{calderon}. 

Consequently, besides Theorem~\ref{thm:qtilde-maxreg}, we only need mapping properties of the Helmholtz projection $\wt{\PP}_2=\wt{\PP}$ between fractional domain spaces in $L^2(\Om)^d$ and $L^2_\si(\Om)$, respectively. These are shown in \cite[Lem. 4.3]{K:hinfty-stokes}.  

We formulate the result on fractional domain spaces, whose proof is much easier here since we have a precise description of domain $D(\wt{A}_p)$ in $\wt{L}^p_\si(\Om)$ for $1<p<\infty$.

\begin{corollary}\label{cor:qtilde-fract}
Let $\Om\subseteq\R^d$ be a uniform $C^{1,1}$-domain. For $1<p<\infty$, $\del>0$, and $\theta\in(0,1)$, we have
$$
 D\big((\del+\wt{A}_p)^\theta\big)=D\big((\del-\wt{\Delta}_p)^\theta\big)^d\cap \wt{L}^p_\si(\Om)
 =\wt{H}^{2\theta,p}_0(\Om)^d\cap\wt{L}^p_\si(\Om),
$$
where $\wt{\Delta}_p$ denotes the Dirichlet Laplacian in $\wt{L}^p(\Om)$, which has domain $\wt{D}^p(\Om)$, and 
$$
  \wt{H}^{2\theta,p}_0(\Om)
  =\left\{\begin{aligned} &H^{2\theta,p}_0(\Om)+H^{2\theta,2}_0(\Om), && 1<p\le 2, \\
  &H^{2\theta,p}_0(\Om)\cap H^{2\theta,2}_0(\Om), && 2<p<\infty.
  \end{aligned}\right.
$$
\end{corollary} 

\begin{proof}
The first equality can be  shown by the arguments in the proof of \cite[Lem. 6]{Giga}. By consistency of resolvents we have
$$
  D\big((\del-\wt{\Delta}_p)^\theta\big)
  =\left\{\begin{aligned} &D\big((\del-\Delta_p)^\theta\big)+ D\big((\del-\Delta_2)^\theta\big), && 1<p\le 2, \\
  &D\big((\del-\Delta_p)^\theta\big)\cap D\big((\del-\Delta_2)^\theta\big), && 2<p<\infty,
  \end{aligned}\right.
$$
and by boundedness of the $H^\infty$-calculus for $\mu+\Delta_p$ we have $D((\mu-\Delta_p)^\theta)=[L^p(\Om),D^p(\Om)]_\theta$.
These interpolation spaces may be determined by localization, see \cite[proof of Cor.~1.2]{K:hinfty-stokes}, and this gives the result.
\end{proof}

\subsection{Relative $L^p$-theory}
Assume that $\Om\subseteq\R^d$ is a uniform $C^3$-domain and that $p\in(1,\infty)$ is such that the Helmholtz decomposition
\begin{align}\label{eq:Lq-HHD}
 L^p(\Om)^d=L^p_\si(\Om)\oplus G^p(\Om)
\end{align}
holds, i.e., that the Helmholtz projection $\PP$ acts boundedly in $L^p(\Om)^d$. It was shown in \cite[Prop. 2.19]{GK} that the set of all such $p\in(1,\infty)$ is an interval containing $2$ and closed under taking dual exponents, and that corresponding Helmholtz projections are consistent. For $p$ in this interval, define the Stokes operator $A_p$ in $L^p_\si(\Om)$ by
$$
 A_pu:=-\PP_p\Delta u,\quad u\in D(A_p):=W^{2,p}(\Om)^d\cap W^{1,p}_0(\Om)^d\cap L^p_\si(\Om).
$$
Under these assumptions,
Gei{\ss}ert, Heck, Hieber, and Sawada showed in \cite[Thm. 1.1]{GHHS} that there exists $\mu\ge0$ such that the operator $\mu+A_p$ has maximal $L^q$-regularity in $L^p_\si(\Om)$ for $1<q<\infty$. By Lutz Weis' theorem mentioned in Subsection~\ref{sub:R-sec} this means that we have the following.

\begin{prop}\label{prop:GHHS-Rsec} 
Let $\Om\subseteq\R^d$ be a uniform $C^3$-domain and $p\in(1,\infty)$ such that \eqref{eq:Lq-HHD} holds.
Then there exists $\mu\ge0$ such that $\mu+A_p$ is $R$-sectorial in $L^p_\si(\Om)$ of some angle $<\frac\pi2$.
\end{prop}

Again, we can use the mapping properties of the Helmholtz projection in $L^2(\Om)^d$ from \cite[Lem. 4.3]{K:hinfty-stokes} for an application of Corollary~\ref{cor:abstract} via Remark~\ref{rem:Cor-sa} and obtain the following, originally proved in \cite[Thm. 1.1]{GK}.

\begin{theorem}\label{thm:GHHS-Hinfty} 
Let $\Om\subseteq\R^d$ be a uniform $C^3$-domain and $p_0\in(1,\infty)$ such that \eqref{eq:Lq-HHD} holds for $p_0$.
Then there exists $\mu\ge0$ such that $\mu+A_p$ has a bounded $H^\infty$-calculus in $L^p_\si(\Om)$ for all $p\in(\min\{p_0,p_0'\},\max\{p_0,p_0'\})$.
\end{theorem}
 
The proof given in \cite{GK} used the result from \cite{KuW:erratum}. Hence, the additional condition on mapping properties of the Helmholtz projection in $L^q$ had to be verified. Thus, the proof given here is simpler.

It was conjectured by M. Geissert that the assertion of \cite[Thm. 1.1]{GHHS} (and hence also of Proposition~\ref{prop:GHHS-Rsec}) remains valid if $\Om$ is just a uniform $C^2$-domain or maybe just a uniform $C^{1,1}$-domain. In this case and with the given proof, we would obtain the same for the assertion of Theorem~\ref{thm:GHHS-Hinfty}. 

\begin{corollary}\label{eq:GHHS-fract}
Under the assumptions of Theorem~\ref{thm:GHHS-Hinfty} we have, for $q\in(\min\{p_0,p_0'\},\max\{p_0,p_0'\})$, 
$\theta\in(0,1)$, and for the $\mu$ from Theorem~\ref{thm:GHHS-Hinfty},
$$
 D\big((\mu+A_p)^\theta\big)=D\big((\mu-\Delta_p)^\theta\big)^d\cap L^p_\si(\Om)=H^{2\theta,p}_0(\Om)^d\cap L^p_\si(\Om),
$$
where $\Delta_p$ denotes the Dirichlet Laplacian in $L^p(\Om)$ with domain $D^p(\Om)$. 
\end{corollary}

The proof is in fact simpler than the proof of Corollary~\ref{cor:qtilde-fract} given above.

%\begin{appendix}
%\section{Auxiliary results}
%\end{appendix}

\end{document}